\documentclass[12pt]{amsart}
\usepackage{a4wide,enumerate,xcolor,graphicx}
\usepackage{amsmath}
\usepackage{scrextend} 
\allowdisplaybreaks

\let\pa\partial
\let\na\nabla
\let\eps\varepsilon
\newcommand{\N}{{\mathbb N}}
\newcommand{\R}{{\mathbb R}}

\newcommand{\diver}{\operatorname{div}}
\newcommand{\dom}{{\mathcal D}}
\newcommand{\dd}{{\mathrm d}}
\newcommand{\F}{{\mathcal F}}
\newcommand{\B}{{\mathcal B}}
\newcommand{\E}{\mathbb{E}}

\newtheorem{theorem}{Theorem}
\newtheorem{lemma}[theorem]{Lemma}

\begin{document}

\title[A coupled model for angiogenesis]{A coupled stochastic
differential reaction-diffusion system for angiogenesis} 

\author[M. Fellner]{Markus Fellner}
\address{Institute of Analysis and Scientific Computing, Vienna University of  
	Technology, Wiedner Hauptstra\ss e 8--10, 1040 Wien, Austria}
\email{markus.fellner@tuwien.ac.at} 

\author[A. J\"ungel]{Ansgar J\"ungel}
\address{Institute of Analysis and Scientific Computing, Vienna University of  
	Technology, Wiedner Hauptstra\ss e 8--10, 1040 Wien, Austria}
\email{juengel@tuwien.ac.at} 

\date{\today}

\thanks{The authors acknowledge partial support from   
the Austrian Science Fund (FWF), grants P33010, W1245, and F65.
This work has received funding from the European 
Research Council (ERC) under the European Union's Horizon 2020 research and 
innovation programme, ERC Advanced Grant no.~101018153.} 

\begin{abstract}
A coupled system of nonlinear mixed-type equations
modeling early stages of angiogenesis is analyzed in a bounded domain. 
The system consists of
stochastic differential equations describing the movement of the positions
of the tip and stalk endothelial cells, 
due to chemotaxis, durotaxis, and random motion;
ordinary differential equations for the volume fractions of the extracellular fluid,
basement membrane, and fibrin matrix; and reaction-diffusion equations for
the concentrations of several proteins involved in the angiogenesis process.
The drift terms of the stochastic differential equations involve the gradients 
of the volume fractions and the concentrations, and the diffusivities in the
reaction-diffusion equations depend nonlocally on the volume fractions, making
the system highly nonlinear. The existence of a unique solution to this system
is proved by using fixed-point arguments and H\"older regularity theory.
Numerical experiments in two space dimensions illustrate the onset of
formation of vessels.
\end{abstract}

 

\maketitle


\section{Introduction}

Angiogenesis is the process of expanding existing blood vessel networks by
sprouting and branching. Its mathematical modeling is important to understand,
for instance, wound healing, inflammation, and tumor growth. 
In this paper, we analyze a variant 
of the continuum-scale model suggested in \cite{BMGV16} that is used to simulate
the early stages of angiogenesis. The model takes into account the dynamics
of the tip (leading) endothelial cells by solving stochastic differential
equations, the influence of various proteins triggering the cell dynamics by
solving reaction-diffusion equations, and the change of some components 
of the extracellular matrix into extracellular fluid 
by solving ordinary differential equations.
Up to our knowledge, this is the first analysis of the model of \cite{BMGV16}. 

Angiogenesis is mainly triggered by local tissue hypoxia 
(low oxygen level in the tissue), which activates
the production of the signal protein 
vascular endothelial growth factor (VEGF). Endothelial
cells, which form a barrier between vessels and tissues, reached by the VEGF
signal initiate the angiogenic program. These cells break out of the vessel wall,
degrade the basement membrane (a thin sheet-like structure providing cell and
tissue support), proliferate, and invade the surrounding tissue while still 
connected with the vessel network. The angiogenic program specifies the
activated endothelial cells into tip cells (cells at the front of the
vascular sprouts) and stalk cells (highly proliferating cells). 
The tip cells lead the sprout towards the source of VEGF, while the stalk cells
proliferate to follow the tip cells supporting sprout elongation; see
Figure \ref{fig.overview}. 

\begin{figure}[ht]
\includegraphics[width=120mm]{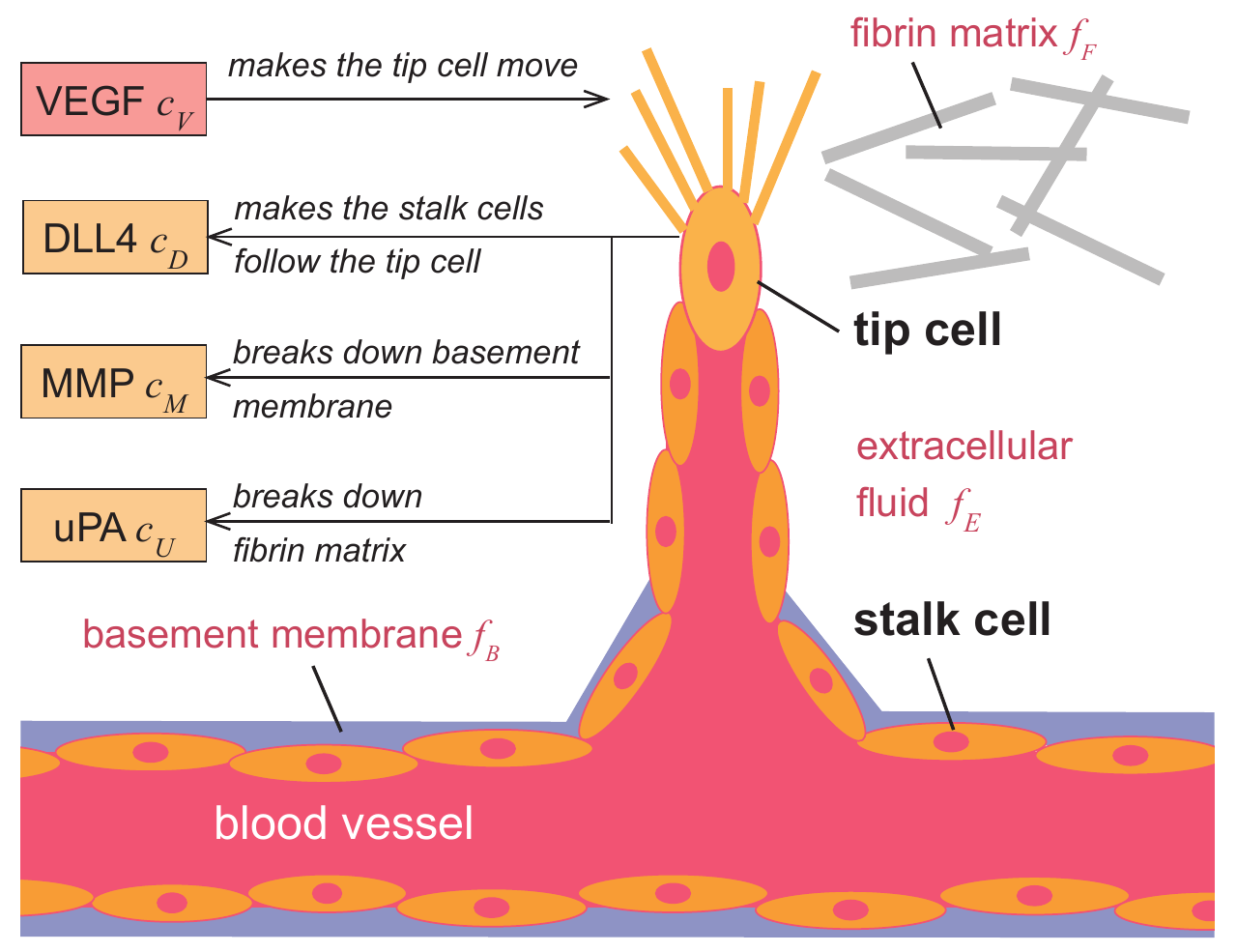}
\caption{Schematic model of sprout formation in a blood vessel. Tip cells are
activated by VEGF and they segrete the proteins DLL4, MMP, and uPA. The vessels
are embedded in the fibrin matrix and extracellular fluid and surrounded by the
basement membrane.}
\label{fig.overview}
\end{figure}

Following \cite{BMGV16}, the tip cells segrete the proteins
delta-like ligand 4 (DLL4), matrix metalloproteinase (MMP), and urokinase
plasminogen activator (uPA). The chemokine DLL4 makes the stalk cells follow the
tip cells, MMP breaks down the basement membrane, and uPA degrades the
fibrin matrix such that cells can move into. There are many other molecular
mechanisms and mediators in the angiogenesis process; see, e.g., \cite{BlGe13,UGEG10}
for details.

\subsection{Model equations}

The unknowns are 
\begin{itemize}
\item the positions $X_1^k(t)$ of the $k$th tip cell and $X_2^k(t)$ of the
$k$th stalk cell;
\item the volume fractions of the basement membrane 
$f_B(x,t)$, the extracellular fluid $f_E(x,t)$, and the fibrin matrix $f_F(x,t)$;
\item the concentrations of the proteins VEGF $c_V(x,t)$, DLL4 $c_D(x,t)$,
MMP $c_M(x,t)$, and uPA $c_U(x,t)$, 
\end{itemize}
where $t\ge 0$ is the time and $x\in\dom\subset\R^3$ the spatial
variable. All unknowns depend additionally on the stochastic variable 
$\omega\in\Omega$, where $\Omega$ is the set of events. 
We assume that the mixture of basement membrane, extracellular fluid, 
and fibrine matrix is saturated, i.e., the volume fractions $f_B$, $f_E$, and $f_F$
sum up to one.
We introduce the vectors $X=(X_i^k)_{i=1,2,k=1,\ldots,N_i}$,
$f=(f_B,f_E,f_F)$, and $c=(c_V,c_D,c_M,c_U)$.

\subsubsection*{Stochastic differential equations} 
The tip cells move according to chemotaxis force, driven by the gradient
of the VEGF concentration, the durotaxis force, driven by the gradient of
the solid fraction $f_S:=f_B+f_F$, and random motion modeling uncertainties. 
The dynamics of $X_i^k(t)$ is assumed to be
governed by the stochastic differential equations (SDEs)
\begin{equation}\label{1.X}
  \dd X_i^k(t) = g_i[c,f](X_i^k,t)\dd t + \sigma_i(X_i^k)\dd W_i^k(t), \quad t>0,
	\quad X_i^k(0) = X_i^0,
\end{equation}
where $i=1,2$, $k=1,\ldots,N_i$ for $N_i\in\N$, $W_i^k(t)$ are Wiener processes,
and the drift terms
\begin{equation}\label{1.g}
\begin{aligned}
  g_1[c,f](X_1,t) &:= \alpha_0 M_1(f_S(X_1),X_1)z_1
	+ \gamma(f_S(X_1))\na c_V(X_1,t)	+ \lambda(f_S(X_1))\na f_S(X_1,t), \\
  g_2[c,f](X_2,t) &:= \alpha_0 M_2(f_S(X_2),X_2)z_2 
	+ \gamma(f_S(X_2))\na c_D(X_2,t)	+ \lambda(f_S(X_2))\na f_S(X_2,t)
\end{aligned}
\end{equation}
include a constant $\alpha_0>0$, the strain energy $M_i$, 
the direction of movement determined by the 
strain energy density $z_i$, 
the chemotaxis force $\na c_V$ in the direction of VEGF (tip cells) and 
$\na c_D$ in the direction of DLL4 (stalk cells), 
and the migration as a result of the durotaxis force $\na f_S$. 
In this paper, we allow for general drift terms by imposing suitable
Lipschitz continuity conditions; see Assumption (A4) below.
In the numerical experiments, we choose the functions $M_i$, $\gamma$,
and $\lambda$ as in Appendix \ref{app.model}.

\subsubsection*{Ordinary differential equations}
The proteins MMP and uPA degrade the basement membrane and fibrin matrix, respectively,
while enhancing the extracellular fluid component. Therefore, the volume fractions
$f_B$, $f_E$, and $f_F$ are determined by the ordinary differential equations (ODEs)
\begin{equation}\label{1.f}
\begin{aligned}
  \frac{\dd f_B}{\dd t} &= -s_Bc_Mf_B, \quad t>0, \quad f_B(0)=f_B^0, \\
	\frac{\dd f_F}{\dd t} &= -s_Fc_Uf_F, \quad t>0, \quad f_F(0)=f_F^0, \\
	\frac{\dd f_E}{\dd t} &= s_Bc_Mf_B+s_Fc_Uf_F, \quad t>0, \quad f_E(0)=1-f_B^0-f_F^0,
\end{aligned}
\end{equation}
where $s_B$, $s_F>0$ are some rate constants.
Note that the last differential equation is redundant because of the 
volume-filling condition $f_B+f_E+f_F=1$. Clearly, equations \eqref{1.f} can
be solved explicitly, giving for $(x,t)\in\dom\times(0,T)$ and
pathwise in $\Omega$ (we omit the argument $\omega$),
\begin{equation}\label{1.fBF}
\begin{aligned}
  f_B(x,t) &= f_B^0(x)\exp\bigg(-s_B\int_0^t c_M(x,s)\dd s\bigg), \\
	f_F(x,t) &= f_F^0(x)\exp\bigg(-s_F\int_0^t c_U(x,s)\dd s\bigg).
	\end{aligned}
\end{equation}

\subsubsection*{Reaction-diffusion equations}
The mass concentrations are modeled by reaction-diffusion equations, describing
the consumption and production of the proteins:
\begin{equation}\label{1.c}
\begin{aligned}
  & \pa_t c_V - \diver(D_V(f)\na c_V) + \alpha_V(x,t) c_V = 0
	\quad\mbox{in }\dom,\ t>0, \\
	& \pa_t c_D - \diver(D_D(f)\na c_D) + \beta_D(x,t) c_D 
    = \alpha_D(x,t)c_V, \\
	& \pa_t c_M - \diver(D_M(f)\na c_M) + s_{M} f_B c_M = \alpha_M(x,t) c_V, \\
	& \pa_t c_U - \diver(D_U(f)\na c_U) + s_{U} f_F c_U = \alpha_U(x,t) c_V,
\end{aligned}
\end{equation}
with initial and no-flux boundary conditions
\begin{equation}\label{1.cbic}
	c_j(0)=c_j^0\quad\mbox{in }\dom, \quad \na c_j\cdot\nu=0\quad \mbox{on }\pa\dom,\
	j=V,D,M,U,
\end{equation}
where the rate terms are given by $\alpha_V=s_V\widetilde\alpha_V$,
$\alpha_j=r_j\widetilde\alpha_j$ for $j=D,M,U$,
$\beta_D=s_D\widetilde\beta_D$, and 
\begin{equation}\label{1.ab}
  \widetilde\alpha_j(x,t) = \sum_{k=1}^{N_1}V_j^k(X_1^k(t)-x), \quad
  \widetilde\beta_D(x,t) = \sum_{k=1}^{N_2}V_D^k(X_2^k(t)-x), 
\end{equation}
for $j=V,D,M,U$ and $V_j^k:\Omega\times\R^3\to\R$ are nonnegative smooth potentials
approximating the delta distribution. The parameters $r_j$ and $s_j$ are positive.
In \cite{BMGV16}, the rate terms are given by delta distributions instead of
smooth potentials. We assume smooth potentials because of regularity issues,
but they can be given by delta-like functions as long as they are smooth.
Indeed, we need $C^{1+\delta}(\overline\dom)$ solutions $c_j$ to solve
the SDEs \eqref{1.X}, and this regularity is not possible when the
source terms of \eqref{1.c} include delta distributions.
As the number of the proteins is typically much larger than the number of tip cells,
the stochastic fluctuations in the concentrations are expected to be much smaller
than those associated with the tip cells, which justifies the macroscopic approach
using reaction-diffusion equations for the concentrations.

In equation \eqref{1.c} for $c_V$, the term $\alpha_Vc_V$ models the consumption
of VEGF along the trajectory of the tip cells. The protein DLL4 is regenerated from
conversion of VEGF, modeled by $\alpha_Dc_V$ along the trajectories
of the tip cells, and consumed by the stalk cells, modeled by $\beta_Dc_D$ 
along the trajectories of the stalk cells. In equation \eqref{1.c} for $c_M$,
the term $s_M f_B c_M$ describes the decay of the MMP concentration with rate
$s_B>0$ as a result of the breakdown of the basement membrane, and 
$\alpha_Mc_V$ models the production of MMP due to conversion from VEGF.
Similary, $s_U f_F c_U$ describes the decay of the uPA concentration, which
breaks down the fibrin matrix, and the protein uPA is regenerated, leading to
the term $\alpha_Uc_V$. 

The diffusivities are given by the mixing rule
$$
  D_j(f) = D_j^B f_B + D_j^E f_E + D_j^F f_F, \quad j=V,D,M,U,
$$
where $D_j^i>0$ for $i=B,E,F$. Then 
\begin{equation}\label{1.lowup}
  0<\min\{D_j^B,D_j^E,D_j^F\}\le D_j(f)\le \max\{D_j^B,D_j^E,D_j^F\},
\end{equation}
and equations \eqref{1.c} are uniformly parabolic. 
Note, however, that equations \eqref{1.c}
are nonlocal and quasilinear, since the diffusivities are determined by the
time integrals of $c_M$ or $c_U$; see \eqref{1.fBF}.

Various biological phenomena are not modeled by our equations.
For instance, we do not include the initiation of sprouting from preexisting parental 
vessels, the branching from a tip cell, and anastomosis (interconnection between
blood vessels). Moreover, 
in contrast to \cite{BMGV16}, we do not allow for the transition between the
phenotypes ``tip cell'' and ``stalk cell'' to simplify the presentation.
On the other hand, we may include further angiogenesis-related proteins,
if the associated reaction-diffusion equations are of the structure \eqref{1.c}.

\subsection{State of the art}

There are several approaches in the literature to model angiogenesis, mostly
in the context of tumor growth.
Cellular automata models divide the computational domain into a discrete set
of lattice points, and endothelial cells move in a discrete way. Such models
are quite flexible, and intra-cellular adherence can be easily implemented,
but their numerical solution is computationally expensive 
when the numbers of cells or molecules 
are large \cite{GrGl92}. In semicontinuous models, the cells are treated as discrete
entities, but their movement is not restricted to any lattice points \cite{ByDr09}.
Continuum-scale models consider cells as averaged quantities, leading to
cell densities whose dynamics is described by partial differential equations;
see, e.g., \cite{BrCh93} for wound healing and \cite{GPM02} for angiogenesis.
Chemotaxis can be modeled in this approach by Keller--Segel-type equations,
which admit global weak solutions in two space dimensions without blowup \cite{CPZ04}.
A hybrid approach was investigated in \cite{CGCCO12}, where the blood vessel
network is implemented on a lattice, tip cells are moving
in a lattice-free way, and other cells are modeled macroscopically as densities.
A continuum-scale approach was chosen in \cite{BMGV16}, which is the basis of
the present paper. The novelty of \cite{BMGV16} is the distinction of tip
and stalk cells and the inclusion of proteins secregated by the tip cells.

In other models, stochastic effects have been included. In \cite{StLa91}, the movement
of the tip cells is modeled by a SDE, with a deterministic part describing 
chemotaxis, and a stochastic part modeling random motion. 
The mean-field limit in a stochastic many-particle system, leading to 
reaction-diffusion equations, was performed in \cite{CaFl19,SGAMB15}. 
We also refer to the reviews \cite{CMA06,Pie08} on further modeling approaches
of angiogenesis.

Numerical simulations of a coupled SDE-PDE model
for the movement of the tip cells and the dynamics of the tumor angiogenesis factor,
fibronectin (a protein of the extracellular matrix), and matrix degrading enzymes
were presented in \cite{CaMo09}. Other SDE-PDE models in the literature are
concerned with the proton dynamics in a tumor \cite{KSS16},
acid-mediated tumor invasion \cite{HSZS18}, and viscoelastic fluids \cite{JLL05}. 
However, only the works \cite{HSZS18,KSS16} treat a genuine coupling between SDEs 
and PDEs. While the model in \cite{HSZS18} also includes a cross-diffusion term
in the equation for the cancer cells, we have simpler reaction-diffusion equations
but with nonlocal diffusivities. Up to our knowledge, the mathematical analysis 
of system \eqref{1.X}, \eqref{1.f}--\eqref{1.cbic} is new.

\subsection{Assumptions and main result}

Let $(\Omega,\F,(\F_t)_{t\ge 0},\mathbb{P})$ be a stochastic basis with a 
complete and right-continuous filtration, and let $(W_i^k(t))_{t\ge 0}$ for $i=1,2$,
$k=1,\ldots,N_i$ be independent standard Wiener processes on $\R^d$ ($d\in\N$)
relative to $(\F_t)_{t\ge 0}$. We write $L^p(\Omega,\F;B)$ for the set of all
$\F$-measurable random variables with values in a Banach space $B$, for which
the $L^p$ norm is finite. Furthermore, let $\dom\subset\R^3$ be a bounded domain
with boundary $\pa\dom\in C^3$ (needed to obtain parabolic regularity; see
Theorem \ref{thm.class}). We set $Q_T=\dom\times(0,T)$.

We write $C^{k+\delta}(\overline\dom)$ with $k\in\N_0$, $\delta\in(0,1)$
for the space of $C^k$ functions $u$ such that the $k$th derivative 
$\mathrm{D}^ku$ is H\"older continuous of index $\delta$. The space of
Lipschitz continuous functions on $\overline\dom$ is denoted by 
$C^{0,1}(\overline\dom)$. For notational convenience, 
we do not distinguigh between the spaces
$C^{k+\delta}(\overline\dom;\R^n)$ and $C^{k+\delta}(\overline\dom)$.
Furthermore, we usually drop the dependence on the variable 
$\omega\in\Omega$ in the ODEs and PDEs, which hold pathwise $\mathbb{P}$-a.s.
Accordingly, we write $c$ instead of $c(\omega,\cdot,\cdot)$ and 
$c(t)$ instead of $c(\omega,\cdot,t)$. Finally, we write ``a.s.'' instead of
``$\mathbb{P}$-a.s.''.

We impose the following assumptions:

\begin{itemize}
\item[(A1)] Initial data: $X^0_i\in L^4(\Omega,\F_0)$ satisfies $X_i^0\in\dom$
a.s.\ ($i=1,2$),
$c^0\in L^\infty(\Omega,\F_0;$ $C^{2+\delta_0}(\overline\dom))$,
$f^0\in L^\infty(\Omega,\F_0;C^{1+\delta_0}(\overline\dom))$ for some $0<\delta_0<1$;
$c_j^0\ge 0$ ($j=V,D,M,U$), $f_i^0\ge 0$ ($i=B,E,F$), and
$f_B^0+f_E^0+f_F^0=1$ in $\dom$ a.s.; $\na c_j^0\cdot\nu=0$
on $\pa\dom$ a.s.
\item[(A2)] Diffusion: $\sigma_i:\Omega\times\dom\times[0,T]\to\R^{3\times d}$
($i=1,2$) is Lipschitz continuous in $x$, has at most linear growth in $x$, 
is progressively measurable, and satisfies $\sigma_i=0$ on $\pa\dom$ a.s.
\item[(A3)] Drift: $g_i[c,f]:\Omega\times\dom\times[0,T]\to\R^3$ ($i=1,2$) for
$f,c\in C^0(\overline\dom\times[0,T])$ is progressively measurable if $(c,f)$ are,
and $g_i[c,f](x,t)=0$ for $x\in\pa\dom$, $t\in[0,T]$.
\item[(A4)] Lipschitz continuity for $g_i$: For $c,c',f,f'\in C^1(\overline\dom
\times[0,T])$, there exists $L_1>0$ such that for 
$(x,t)\in\overline\dom\times[0,T]$ and $i=1,2$,
\begin{align*}
  \big|g_i[c,f](x,t)-g_i[c',f'](x,t)\big|
	&\le L_1\big(1 + \|c\|_{L^\infty(0,t;C^1(\overline\dom))} 
	+ \|f\|_{L^\infty(0,t;C^1(\overline\dom))}\big) \\
	&\phantom{xx}{}\times
	\big(\|c(t)-c'(t)\|_{C^1(\overline\dom)} + \|f(t)-f(t')\|_{C^1(\overline\dom)}\big).
\end{align*}
Furthermore, for $c,f\in L^\infty(0,T;W^{2,\infty}(\dom))$, there exists $L_2>0$
such that for $(x,t),(x',t)\in\overline\dom\times[0,T]$ and $i=1,2$,
\begin{align*}
  \big|g_i[c,f](x,t)-g_i[c,f](x',t)\big|
	&\le L_2\big(1 + \|c\|_{L^\infty(0,T;W^{2,\infty}(\dom))}
	+ \|f\|_{L^\infty(0,T;W^{2,\infty}(\dom))}\big) \\
	&\phantom{xx}{}\times|x-x'|.
\end{align*}
\item[(A5)] Potentials: $V_j^k\in C^{0,1}(\R^3)$ for $j=V,D,M,U$, $k=1,\ldots,N_i$ 
are nonnegative functions.
\end{itemize}

Let us discuss these assumptions. Since we need $C^{1+\delta}$ solutions $(c,f)$
to obtain H\"older continuous coefficients of the SDEs (which ensures their
solvability), we need some regularity conditions on the initial data
in Assumption (A1).
Accordingly, $\na c_j^0\cdot\nu=0$ on $\pa\dom$ is a compatibility condition
needed for such a regularity result. In Assumption (A1), we impose the
volume-filling condition initially, $f_B^0+f_E^0+f_F^0=1$ in $\dom$.
Equations \eqref{1.f} then show that this condition is satisfied for all time.
The conditions $\sigma_i=0$ and $g_i[c,f]=0$ on $\pa\dom$ in Assumptions
(A2) and (A3), respectively, guarantee that 
the tip and stalk cells do not leave the domain $\dom$.
The conditions on $\sigma_i$ in Assumption (A2) and the Lipschitz continuity
of $g_i[c,f]$ in Assumption (A4) are standard hypotheses to apply existence
results for \eqref{1.X}. Note that $g_i[c,f]$ in example \eqref{1.g} satisfies
Assumption (A4). Finally, Assumption (A5) is a simplification to ensure the
parabolic regularity results needed, in turn, for the solvability of \eqref{1.X}.

\begin{theorem}[Global existence]\label{thm.ex}
Let Assumptions (A1)--(A5) hold. Then there exists a unique solution $(X,c,f)$
to \eqref{1.X}, \eqref{1.f}--\eqref{1.ab} such that 
\begin{itemize}
\item $X=(X_i^k)_{i=1,2,k=1,\ldots,N_i}$ is a strong solution to \eqref{1.X};
\item $f=(f_B,f_E,f_F)$ solves \eqref{1.f} pathwise a.s.\ in the sense of \eqref{1.fBF},
where $f_i\in C^0([0,T];$ $L^2(\dom))\cap L^\infty(Q_T)$;
\item $c=(c_V,c_D,c_M,c_U)$ is a classical solution to \eqref{1.c}--\eqref{1.cbic}
pathwise a.s.
\end{itemize}
\end{theorem}

A strong solution $X$ to \eqref{1.X} means, according to \cite[Theorem 3.1.1]{LiRo15},
that $(X_i^k(t))_{t\ge 0}$ is an a.s.\ continuous $(\F_t)$-adapted
process such that for all $t\in[0,T]$,
\begin{equation}\label{1.intX}
  X_i^k(t) = X_i^0 + \int_0^t g_i[c,f](X_i^k(s),s)\dd s
	+ \int_0^t\sigma_i(X_i^k(s))\dd W_i^k(s)\quad a.s.
\end{equation}

\subsection{Strategy of the proof}

The proof of Theorem \ref{thm.ex} is based on a variant of
Banach's fixed-point theorem \cite[Theorem 2.4]{Lat14}
yielding global solutions. 
Let $\widetilde{X}$ be a stochastic process with a.s.\ H\"older continuous
paths and values in $\dom$ a.s. More precisely, $\widetilde{X}_i^k\in Y_R(0,T;\dom)$,
where $Y_R(0,T;\dom)\subset C^{1/2}([0,T];L^4(\Omega))$ is defined in 
\eqref{def.YR} below. Then $\alpha_j$ and $\beta_j$ are H\"older continuous in
$\overline\dom\times[0,T]$ a.s.\ as a function of $\widetilde{X}$.

As the first step, we prove the solvability of the ODEs 
\eqref{1.f} with $c_j$ only lying in the
space $L^\infty(0,T;L^2(\dom))$. This is achieved by an approximation argument,
assuming a sequence $c_j^{(k)}\in L^\infty(Q_T)$ and then passing to the
limit. This is possible since the solutions $f_i$, represented in \eqref{1.fBF},
are bounded.

Second, we show that there exists a weak solution $c$, which is defined pathwise for
$\omega\in\Omega$, to the reaction-diffusion
equations \eqref{1.c} by using a Galerkin method and the fixed-point theorem
of Schauder for fixed $\alpha_j$ and $\beta_j$. The regularity of $c$ can be
improved step by step.
Moser's iteration method shows that the weak solution is in fact bounded, and
a general regularity result for evolution equations yields $\pa_t c\in L^2(Q_T)$.
Then, interpreting equations \eqref{1.c} as elliptic equations with right-hand
side $\pa_t c\in L^2(Q_T)$, we conclude the H\"older continuity of $c(t)$
for any fixed $t\in(0,T)$. Thus, the diffusivities are H\"older continuous,
and we infer $C^{1+\delta}(\overline\dom)$ and $W^{2,\infty}(\dom)$ solutions $c$. 
We conclude classical solvability from the regularity results of 
Lady\v{z}enskaya et al.\ \cite{LSU68}.

In the third step, we solve the SDEs \eqref{1.X}. The functions
$(c,f)$ have H\"older continuous gradients, and we show that $(c,f)$ and
$(\na c,\na f)$ are progressively measurable. Therefore, together with Assumption
(A4), the conditions of the existence theorem of \cite[Theorem 3.1.1]{LiRo15}
are satisfied, and we obtain a solution $X$ to \eqref{1.X} 
in the sense \eqref{1.intX}.

Fourth, defining a suitable complete metric space $Y_R(0,T;\dom)$, endowed with
the $C^0([0,T];$ $L^4(\dom))$ norm, this defines the
fixed-point operator $\Phi:\widetilde{X}\mapsto X$ on $Y_R(0,T;\dom)$. 
The fixed-point operator can be written as the contatenation
$$
  \Phi:\widetilde{X}\mapsto (\alpha,\beta_D)\mapsto (c,f)\mapsto X,
$$
where $\alpha=(\alpha_V,\alpha_D,\alpha_M,\alpha_U)$.
It remains to show that $\Phi$ is a self-mapping and a contraction, 
which is possible for a sufficiently large $R>0$. 
In fact, we show that for any $n\in\N$,
$$
  \sup_{0<s<t}\big(\E|\Phi^n(X(t))-\Phi^n(X'(t))|^4\big)^{1/4}
	\le c_n\sup_{0<s<t}\big(\E|X(t)-X'(t)|^4\big)^{1/4}
$$
for all $X$, $X'\in Y_R(0,T;\dom)$, where $c_n\to 0$ as $n\to\infty$.
It follows from the variant of Banach's fixed-point theorem in 
\cite[Theorem 2.4]{Lat14} 
that $\Phi$ has a unique fixed point, which gives a unique solution
$(X,c,f)$ to \eqref{1.X}, \eqref{1.f}--\eqref{1.ab}. 

The paper is organized as follows. The solvability of the ODEs \eqref{1.f}
for $L^2(\dom)$ coefficients and a stability estimate are shown in Section \ref{sec.f}.
In Section \ref{sec.c}, the existence of a unique classical solution to the
reaction-diffusion equations  
\eqref{1.c}--\eqref{1.cbic} and some stability results are proved. 
The progressive measurability of the solutions to \eqref{1.f} and \eqref{1.c}
as well as the solvability of the SDEs \eqref{1.X} is verified in Section \ref{sec.X}.
Based on these results, Theorem \ref{thm.ex} is proved in Section \ref{sec.ex}.
Some numerical experiments are illustrated in Section \ref{sec.num}, showing
stalk cells following the tip cells and forming a premature sprout.
Appendix \ref{sec.app} summarizes some regularity results for elliptic and
parabolic equations used in this work, and Appendix \ref{app.model} collects the numerical values of the various parameters used in the numerical experiments.


\section{Solution of ODEs in Bochner spaces}\label{sec.f}

We prove a result for ODEs in some Bochner space, which is needed for
the solution of \eqref{1.f} when the concentrations $c_j$ are only
$L^2(\dom)$ functions.

\begin{lemma}
Let $u\in L^\infty(0,T;L^2(\dom))$ and $g^0\in L^\infty(\dom)$ be 
nonnegative functions. Then
\begin{equation}\label{2.g}
  \frac{\dd g}{\dd t} = -ug, \quad t>0, \quad g(0) = g^0\quad \mbox{a.e. in }\dom,
\end{equation}
has a unique solution $g\in C^0([0,T];L^2(\dom))\cap L^\infty(Q_T)$
satisfying $0\le g\le\|g^0\|_{L^\infty(\dom)}$ a.e.\ in $Q_T$.
\end{lemma}

\begin{proof}
Set $[u]_k:=\max\{k,u\}=u(t)-(u(t)-k)^+$ for $k\ge 0$, where $z^+=\max\{0,z\}$. 
We claim that $[u]_k\in L^\infty(Q_T)$
has the properties $[u(t)]_k\to u(t)$ strongly in $L^2(\dom)$ as $k\to\infty$
for a.e.\ $t\in[0,T]$ and 
\begin{equation}\label{2.ukj}
  \|[u(t)]_k-[u(t)]_j\|_{L^2(\dom)} \le \|[u(t)]_j-u(t)\|_{L^2(\dom)} \le C
	\quad\mbox{a.e. in }[0,T]
\end{equation}
for $0\le j\le k$, where $C>0$ does not depend on $k$ or $j$. Indeed, the first
inequality follows from
\begin{align*}
  \|[u&(t)]_k-[u(t)]_j\|_{L^2(\dom)}^2
	= \int_\dom\big((u(t)-k)^+ - (u(t)-j)^+\big)^2\dd x \\
	&= \int_{\{u(t)\ge k\}}\big((u(t)-k)^+ - (u(t)-j)^+\big)^2\dd x
	+ \int_{\{j\le u(t)\le k\}}[(u(t)-j)^+]^2 \dd x \\
	&= \int_{\{u(t)\ge k\}}(k-j)^2\dd x + \int_{\{j\le u(t)\le k\}}(u(t)-j)^2 \dd x \\
	&\le \int_{\{u(t)\ge k\}}(u(t)-j)^2\dd x 
	+ \int_{\{j\le u(t)\le k\}}(u(t)-j)^2 \dd x \\
	&= \int_\dom[(u(t)-j)^+]^2 \dd x = \|[u(t)]_j-u(t)\|_{L^2(\dom)}^2.
\end{align*}
Furthermore, we have
$$
  \|[u(t)]_j-u(t)\|_{L^2(\dom)}^2 = \int_\dom[(u(t)-j)^+]^2\dd x
	\le \|u\|^2_{L^\infty(0,T;L^2(\dom))} \le C. 
$$
The function $u(t)^2$ is an integrable upper
bound for $(u(t)-k)^+$. Furthermore, $[u(t)]_k-u(t)=-(u(t)-k)^+$ converges
to zero a.e.\ in $\dom$ as $k\to\infty$. Therefore, we conclude from 
dominated convergence that $[u(t)]_k-u(t)\to 0$ strongly in $L^2(\dom)$.
This proves the claim.

Next, we consider the differential equation
\begin{equation}\label{2.gk}
  \frac{\dd g_k}{\dd t} = -[u]_kg_k, \quad t>0, \quad g_k(0) = g^0\quad 
	\mbox{a.e. in }\dom.
\end{equation}
Since $[u]_k$ is bounded, there exists a unique solution 
$g_k\in C^0([0,T];L^\infty(\dom))$, given by
$$
  g_k(t) = g^0\exp\bigg(-\int_0^t[u(s)]_k\dd s\bigg), \quad 0<t<T.
$$
Clearly, we have $0<g_k(t)\le K:=\|g^0\|_{L^\infty(\dom)}$. 
We want to prove that $(g_k)$ converges as $k\to\infty$ to a solution of the
original equation. It follows that
\begin{align*}
  \frac12&\frac{\dd}{\dd t}\|g_k-g_j\|_{L^2(\dom)}^2 
	= -\int_\dom([u]_kg_k-[u]_jg_j)(g_k-g_j)\dd x \\
	&= -\int_\dom[u]_k(g_k-g_j)^2\dd x - \int_\dom([u]_k-[u]_j)g_j(g_k-g_j)\dd x \\
	&\le \frac{K}{2}\|[u]_k-[u]_j\|_{L^2(\dom)}^2 + \frac{K}{2}\|g_k-g_j\|_{L^2(\dom)}^2.
\end{align*}
We conclude from Gronwall's lemma and \eqref{2.ukj} that
\begin{align*}
  \|g_k(t)-g_j(t)\|_{L^2(\dom)}^2 
	&\le \int_0^t e^{K(t-s)}\|[u(s)]_k-[u(s)]_j\|_{L^2(\dom)}^2\dd s \\
	&\le e^{Kt}\int_0^t\|[u(s)]_j - u(s)\|_{L^2(\dom)}^2\dd s.
\end{align*}
Because of $\|[u(s)]_j - u(s)\|_{L^2(\dom)}^2\to 0$ as $j\to\infty$ and
the uniform upper bound for $[u(s)]_j-u(s)$, the dominated convergence theorem
implies that, for any $t\in[0,T]$,
$$
  \|g_k(t)-g_j(t)\|_{L^2(\dom)}^2 
	\le e^{KT}\int_0^T\|[u(s)]_j - u(s)\|_{L^2(\dom)}^2\dd s\to 0
	\quad\mbox{as }j,k\to\infty.
$$
Thus, $(g_k(t))$ is a Cauchy sequence for every $t\in[0,T]$ and we infer that
$g_k(t)\to g(t)$ in $L^2(\dom)$. Because of the
$L^\infty(0,T;L^2(\dom))$ bound for $(g_k)$ and dominated convergence again,
$g_k\to g$ in $L^2(Q_T)$, where $g\in L^\infty(0,T;L^2(\dom))$.
There exists a subsequence, which is not relabeled, such that $g_k(t)\to g(t)$
a.e.\ in $\dom$, for any $t\in[0,T]$. We deduce from $g_k(t)\le K$ that
$g(t)\le K$ for all $t\in[0,T]$. This shows that $g\in L^\infty(Q_T)$.
Writing \eqref{2.gk} as an integral equation and performing the limit $k\to\infty$,
the previous convergence results show that $g$ solves \eqref{2.g}.
\end{proof}

We also need the following stability result, which relates the difference $f_1-f_2$
of solutions to \eqref{1.f} with the difference $c_1-c_2$ of solutions to \eqref{1.c}.

\begin{lemma}\label{lem.ODEdiff}
Let $u_1$, $u_2\in L^\infty(0,T;W^{k,\infty}(\dom))$ with $k\in\N_0$, $g_1^0$,
$g_2^0\in W^{k,\infty}(\dom)$, and let
$g_1$, $g_2\in C^0([0,T];W^{k,\infty}(\dom))$ be the unique solutions to
$$
  \frac{\dd g_i}{\dd t} = -u_ig_i,\quad 0<t<T, \quad g_i(0)=g_i^0 \quad\mbox{a.e. in }
	\dom,\ i=1,2.
$$
Then there exists $C>0$, only depending on $T$, the $L^\infty(0,T;W^{k,\infty}(\dom))$
norm of $u_i$, and the $W^{k,\infty}(\dom)$ norm of $g_i^0$, such that for $p>1$,
$$
  \|g_1-g_2\|_{W^{k,p}(\dom)} \le C\big(\|g_1^0-g_2^0\|_{W^{k,p}(\dom)}
	+ \|u_1-u_2\|_{L^1(0,T;W^{k,p}(\dom))}\big).
$$
\end{lemma}

\begin{proof}
The regularity of $g_i$ follows from the explicit formula and the regularity for
$g_i^0$ and $u_i$. Furthermore, taking the $W^{k,p}(\dom)$ norm of
$$
  g_1(t)-g_2(t) = g_1^0-g_2^0 - \int_0^t \big(u_1(g_1-g_2)+(u_1-u_2)g_2\big)\dd x,
$$
the result follows from the regularity $u_1\in L^\infty(0,T;W^{k,\infty}(\dom))$
and $g_2\in W^{k,\infty}(\dom)$.
\end{proof}


\section{Solution of the reaction-diffusion equations}\label{sec.c}

We show the existence of a weak solution to \eqref{1.c}--\eqref{1.cbic} for given
$\alpha$, $\beta_D$ and then prove that this solution is in fact a classical one.

\subsection{Existence of weak solutions}

The existence of solutions to the reaction-diffusion equations \eqref{1.c} is shown by 
using the Galerkin method and the fixed-point theorem of Schauder.

\begin{lemma}[Existence for \eqref{1.c}]
Let $\alpha_j$, $\beta_D\in L^\infty(Q_T)$, $\alpha_j\ge 0$, $\beta_D\ge 0$ 
for $j=V,D,M,U$, $c^0\in L^2(\dom)$, and $f\in L^\infty(\dom)$ satisfying
$f_B+f_E+f_F=1$ in $\dom$. Then there exists a weak solution
$c=(c_ V,c_D,c_M,c_U)$ to \eqref{1.c}--\eqref{1.cbic}.
\end{lemma}

\begin{proof}
Let $(e_\ell)_{\ell\in\N}$ be the sequence of 
eigenfunctions of the Laplace operator with homogeneous
Neumann boundary conditions. They form an orthonormal basis of $L^2(\dom)$ and are
orthogonal with respect to the $H^1(\dom)$ norm. 
The boundary regularity $\pa\dom\in C^3$
guarantees that $e_\ell\in H^3(\dom)\hookrightarrow C^1(\overline\dom)$; 
see Theorem \ref{thm.ellip} in the Appendix. 
Let $E_N=\operatorname{span}\{e_1,\ldots,e_N\}$ for $N\in\N$.
We wish to find a solution $c^N\in C^0([0,T];E_N)$ to the finite-dimensional
problem
\begin{equation}\label{3.cN}
\begin{aligned}
  (\pa_t c_V^N,\phi_V)_{L^2} + (D_V(f^N)\na c_V^N,\na\phi_V)_{L^2}
	&= -(\alpha_V c_V^N,\phi_V)_{L^2}, \\
  (\pa_t c_D^N,\phi_D)_{L^2} + (D_D(f^N)\na c_D^N,\na\phi_D)_{L^2}
	&= (\alpha_D c_V^N,\phi_D)_{L^2} - (\beta_D c_D^N,\phi_D)_{L^2}, \\
	(\pa_t c_M^N,\phi_M)_{L^2} + (D_M(f^N)\na c_M^N,\na\phi_M)_{L^2}
	&= (\alpha_M c_V^N,\phi_M)_{L^2} - s_{M}(f_B^N c_M^N,\phi_M)_{L^2}, \\
	(\pa_t c_U^N,\phi_U)_{L^2} + (D_U(f^N)\na c_U^N,\na\phi_U)_{L^2}
	&= (\alpha_U c_V^N,\phi_U)_{L^2} - s_{U}(f_F^N c_U^N,\phi_U)_{L^2}
\end{aligned}
\end{equation}
for all $\phi=(\phi_V,\phi_D,\phi_M,\phi_U)\in E_N^4$, where
$(\cdot,\cdot)_{L^2}$ denotes the inner product in $L^2(\dom)$, together 
with the initial condition $c^N(0)=c^{0,N}$. Here, $c^{0,N}\in E_N$
satisfies $c^{0,N}\to c^0$ in $L^2(\dom)$ as $N\to\infty$. 
The function $f^N=(f_B^N,f_F^N,f_E^N)$
is the solution to \eqref{1.f} with $c_j$ replaced by $(c_j^N)^+:=\max\{0,c_j^N\}$.
An integration gives
$$
  f_B^N(t) = f^0_B\exp\bigg(-s_B\int_0^t (c_M^N)^+(s)\dd s\bigg), \quad
	f_F(t) = f^0_F\exp\bigg(-s_F\int_0^t (c_U^N)^+(s)\dd s\bigg),
$$
and $f_E^N(t) := 1-f_B^N(t)-f_F^N(t)$ for $0<t<T$.

To solve \eqref{3.cN}, we consider first a linearized version. Let
$\widehat{c}^N\in C^0([0,T];E_N)$ be given. Then
\begin{equation}\label{3.lin}
\begin{aligned}
  (\pa_t c_V^N,\phi_V)_{L^2} + (D_V(\widehat{f}^N)\na c_V^N,\na\phi_V)_{L^2}
	&= -(\alpha_V c_V^N,\phi_V)_{L^2}, \\
  (\pa_t c_D^N,\phi_D)_{L^2} + (D_D(\widehat{f}^N)\na c_D^N,\na\phi_D)_{L^2}
	&= (\alpha_D c_V^N,\phi_D)_{L^2} - (\beta_D c_D^N,\phi_D)_{L^2}, \\
	(\pa_t c_M^N,\phi_M)_{L^2} + (D_M(\widehat{f}^N)\na c_M^N,\na\phi_M)_{L^2}
	&= (\alpha_M c_V^N,\phi_M)_{L^2} - s_{M}(\widehat{f}_B^N c_M^N,\phi_M)_{L^2}, \\
	(\pa_t c_U^N,\phi_U)_{L^2} + (D_U(\widehat{f}^N)\na c_U^N,\na\phi_U)_{L^2}
	&= (\alpha_U c_V^N,\phi_U)_{L^2} - s_{U}(\widehat{f}_F^N c_U^N,\phi_U)_{L^2}
\end{aligned}
\end{equation}
for any $\phi\in E_N^4$, with the initial condition $c^N(0)=c^{0,N}$,
\begin{equation*}
  \widehat{f}_B^N(t) = f^0_B\exp\bigg(-s_B\int_0^t (\widehat{c}_M^N)^+(s)\dd s\bigg), 
	\quad
	\widehat{f}_F^N(t) = f^0_F\exp\bigg(-s_F\int_0^t (\widehat{c}_U^N)^+(s)\dd s\bigg),
\end{equation*}
and $\widehat{f}_E^N(t) := 1-\widehat{f}_B^N(t)-\widehat{f}_F^N(t)$ for $0<t<T$,
is a system of linear ordinary differential equations in the variables $d_{j\ell}^N$,
where $c_j^N(t) = \sum_{\ell=1}^N d_{j\ell}^N(t)e_\ell$, $j=V,D,M,U$.
By standard ODE theory, there exists a unique solution $c^N(t)$ for $t\in[0,T]$.
Taking $\phi_j=c_j^N$ for $j=V,D,M,U$ as test functions in \eqref{3.lin}
and using the positive lower and upper bounds \eqref{1.lowup}
of the diffusivities as well as the $L^\infty$ bound for $\widehat{f}$, 
standard computations lead to the estimates
\begin{equation}\label{3.est}
\begin{aligned}
  \|c^N\|_{L^\infty(0,T;L^2(\dom))} + \|c^N\|_{L^2(0,T;H^1(\dom))}
	+ \|\pa_t c^N\|_{L^2(0,T;H^1(\dom)')} &\le C(T,c^0), \\
  \|c^N\|_{L^\infty(0,T;H^1(\dom))} &\le C(N).
\end{aligned}
\end{equation}

In the following, the notation $(E_N;\|\cdot\|_Z)$ means that $E_N$ is endowed 
with the norm of the Banach space $Z$, requiring that $E_N\subset Z$.

We wish to define a fixed-point operator $\widehat{c}^N\mapsto c^N$. 
For this, we introduce the sets
\begin{align*}
  Y &= \big\{u\in L^\infty(0,T;(E_N;\|\cdot\|_{H^1(\dom)})):
	\pa_t u\in L^2(0,T;(E_N;\|\cdot\|_{H^1(\dom)})')\big\}, \\
	Y_b &= \big\{u\in Y: \|u\|_{L^\infty(0,T;H^1(\dom))} 
	+ \|\pa_t u\|_{L^2(0,T;H^1(\dom)')} \le C(T,c^0)+C(N)\big\},
\end{align*}
and $W$ is defined as the closure of $Y_b^4$ with respect to the
$C^0([0,T];L^2(\dom))$ norm. We claim that the embedding $W\hookrightarrow
C^0([0,T];(E_N;L^2(\dom)))$ is compact. Indeed, $(E_N;H^1(\dom))$ is
compactly embedded into $(E_N;L^2(\dom))$. By the Aubin--Lions lemma,
$Y$ is compactly embedded into $C^0([0,T];(E_N;L^2(\dom)))$. Consequently,
$W=\overline{Y_b}^4$ is compactly embedded into $C^0([0,T];(E_N;L^2(\dom)))$.

This defines the fixed-point operator $S:W\to C^0([0,T];(E_N;L^2(\dom)))$.
Estimates \eqref{3.est} show that $S(W)\subset Y_b^4\subset W$. Furthermore,
we can prove by standard arguments that $S:W\to W$ is continuous.
We deduce from Schauder's fixed-point theorem that 
there exists a fixed point $c^N$ of $S$ satisfying \eqref{3.cN}. 

Now, we perform the limit $N\to\infty$.
Estimates \eqref{3.est} allow us to apply the Aubin--Lions lemma, giving the
existence of a subsequence (not relabeled) such that, as $N\to\infty$,
$$
  c^N\to c\quad\mbox{strongly in }L^2(0,T;L^2(\dom))
$$
for some $c\in L^2(0,T;H^1(\dom))\cap H^1(0,T;H^1(\dom)')$. Moreover,
\begin{align*}
  c^N\rightharpoonup c &\quad\mbox{weakly in }L^2(0,T;H^1(\dom)), \\
	\pa_t c^N\rightharpoonup \pa_t c &\quad\mbox{weakly in }L^2(0,T;H^1(\dom)').
\end{align*}

We define
$$
  f_B(t) = f_B^0\exp\bigg(-s_B\int_0^t c_M^+(s)\dd s\bigg), \quad
	f_F(t) = f_F^0\exp\bigg(-s_F\int_0^t c_U^+(s)\dd s\bigg),
$$
and $f_E(t):=1-f_B(t)-f_F(t)$ for $0<t<T$. 
Lemma \ref{lem.ODEdiff} implies that
$$
  \|f_B^N-f_B\|_{L^2(Q_T)}\le C(T)\|c_M^N-c_M\|_{L^2(Q_T)}\to 0\quad\mbox{as }
	N\to\infty.
$$
We infer from the linear dependence of $D_j(f)$ on $f$ that
$D_j(f^N)-D_j(f)\to 0$ strongly in $L^2(Q_T)$ for $j=V,D,M,U$. This shows that
$D_j(f^N)\na c_j^N\to D_j(f)\na c_j$ weakly in $L^1(Q_T)$
for $j=V,D,M,U$ and $s_{M} f_B^N c_M^N\to s_{M} f_B c_M$,
$s_{U} f_F^N c_U^N\to s_{U} f_F c_U$ strongly in $L^1(Q_T)$.
Thus, passing to the limit $N\to\infty$ in equation \eqref{3.cN} for $c_M^N$,
$$
  \langle\pa_t c_M,\phi\rangle + (D_M(f)\na c_M,\na\phi)_{L^2}
	= (\alpha_M c_M,\phi)_{L^2} - s_{M}(f_B c_M,\phi)_{L^2}
$$
for all test functions $\phi(x,t)=\sum_{k=1}^N t^k e_k\in C^1(\overline{\dom}\times
[0,T])$ and similarly for the other equations in \eqref{3.cN}.
Here, $\langle\cdot,\cdot\rangle$ denotes the dual product between $H^1(\dom)$ and
$H^1(\dom)'$. Since the class of functions $\phi$ of this type is dense in
$L^2(0,T;H^1(\dom))\cap H^1(0,T;H^1(\dom)')$ \cite[Prop.~23.23iii]{Zei90},
we conclude that $(c,f)$ is a weak solution to \eqref{1.f}--\eqref{1.cbic}.

It remains to verify that the limit concentration $c$ is nonnegative componentwise.
We use $c_V^-=\min\{0,c_V\}$ as a test function in the weak formulation of
equation \eqref{1.c} for $c_V$:
$$
  \frac12\frac{\dd}{\dd t}\int_\dom(c_V^-)^2\dd x
	+ \int_\dom D_V(f)|\na c_V^-|^2\dd x = -\int_\dom\alpha_V(c_V^-)^2\dd x \le 0,
$$
which yields $c_V^-(t)=0$ and $c_V(t)\ge 0$ in $\dom$, $t>0$.
Next, we use $c_D^-$ as a test function in the weak formulation of
equation \eqref{1.c} for $c_D$:
$$
  \frac12\frac{\dd}{\dd t}\int_\dom(c_D^-)^2\dd x
	+ \int_\dom D_D(f)|\na c_D^-|^2\dd x = -\int_\dom\beta_D(c_D^-)^2\dd x
	+ \int_\dom\alpha_Dc_V c_D^-\dd x\le 0,
$$
since $c_V\ge 0$ and $c_D^-\le 0$. We infer that $c_D(t)\ge 0$ in $\dom$, $t>0$.
In a similar way, we prove that $c_M(t)\ge 0$ and $c_U(t)\ge 0$ in $\dom$, $t>0$.
This finishes the proof.
\end{proof}


\subsection{Regularity}

The existence theory for SDEs requires more regularity for the solution $c$ to 
\eqref{1.c}. First, we prove $L^\infty$ bounds. We suppose that Assumptions
(A1)--(A5) hold in this section.

\begin{lemma}
Let $c$ be a weak solution to \eqref{1.c}--\eqref{1.cbic}. Then $c\in L^\infty(Q_T)$
and it holds for all $0<t<T$ that 
$\|c_V(t)\|_{L^\infty(\dom)}\le\|c_V^0\|_{L^\infty(\dom)}$ and
$$
  \|c_j(t)\|_{L^\infty(\dom)}\le e^t\big(\|c_j^0\|_{L^\infty(\dom)}
	+ \|\alpha_j\|_{L^\infty(0,T;L^\infty(\dom))}\|c_V^0\|_{L^\infty(\dom)}\big),
	\quad j=D,M,U.
$$
\end{lemma}


\begin{proof}
First, we use $(c_V-K)^+$ with $K:=\|c_V^0\|_{L^\infty(\dom)}$ as a test function
in the weak formulation of equation \eqref{1.c} for $c_V$:
$$
  \frac12\frac{\dd}{\dd t}\int_\dom[(c_V-K)^+]^2\dd x
	+ \int_{\dom} D_V(f)|\na(c_V-K)^+|^2\dd x 
	= -\int_\dom\alpha_V c_V(c_V-K)^+\dd x\le 0.
$$
We conclude that $c_V(t)\le K$ in $\dom$ for $t>0$. 

Second, we show that $c_D$ is bounded.  For this, set $M(t)=M_0 e^t$, where
$M_0=\|c_D^0\|_{L^\infty(\dom)}	+ \|\alpha_D\|_{L^\infty(0,T;L^\infty(\dom))}
\|c_V^0\|_{L^\infty(\dom)}$. Then $(c_D(0)-M)^+=0$ and, choosing $(c_D-M)^+$ 
as a test function in the weak formulation of equation \eqref{1.c} for $c_D$,
\begin{align*}
  \frac12&\frac{\dd}{\dd t}\int_\dom[(c_D-M)^+]^2\dd x
	+ \int_{\dom} D_D(f)|\na(c_D-M)^+|^2 \dd x \\
	&= -\int_\dom\pa_t M(c_D-M)^+\dd x + \int_\dom(\alpha_D c_V - \beta_D c_D)
	(c_D-M)^+\dd x \\
	&\le \int_\dom\big(-M_0+\|\alpha_D\|_{L^\infty(0,T;L^\infty(\dom))}
	\|c_V^0\|_{L^\infty(\dom)}\big)(c_D-M)^+\dd x \le 0,
\end{align*}
where we used $\beta_D\ge 0$, 
and the last inequality follows from the choice of $M_0$. This shows that
$c_D(t)\le M_0 e^t$ in $\dom$, $t>0$. The bounds for $c_M$ and $c_U$ are
shown in an analogous way.
\end{proof}

Next, we prove that the solution $c(t)$ is H\"older continuous.

\begin{lemma}\label{lem.hoelder}
Let $c$ be a
weak solution to \eqref{1.c}--\eqref{1.cbic}. We suppose that there exists
$\Lambda>0$ such that for a.e.\ $0<t<T$,
\begin{equation*}
  \|\alpha_j(t)\|_{L^\infty(\dom)} + \|\beta_D(t)\|_{L^\infty(\dom)} \le\Lambda,
	\quad j=V,D,M,U. 
\end{equation*}
Then there exists $\delta>0$ such that for $0<t<T$,
$$
  \|\pa_t c\|_{L^2(Q_T)} \le C_2, \quad
	\|c(t)\|_{C^\delta(\overline\dom)} \le C_\delta\big(\|c(t)\|_{L^2(\dom)}
	+ \|\pa_t c(t)\|_{L^2(\dom)}\big),
$$
where $C_2>0$ depends on the $L^2(\dom)$ norm of $c^0$, the $L^\infty(\dom)$ norm
of $f^0$, and $\Lambda$, and 
$\delta$, $C_\delta$ depend on the lower and upper bounds \eqref{1.lowup} for 
$D_j$ and $\Lambda$.
\end{lemma}

\begin{proof}
The $L^2(Q_T)$ bound for $\pa_t c$ follows immediately from 
Theorem \ref{thm.pat} in the Appendix. The H\"older estimate follows from
\cite[Prop.~3.6]{Nit11}. Indeed, we interpret equation \eqref{1.c} for $c_V$,
$$
  \diver(D_V(f)\na c_V) + \alpha_Vc_V = -\pa_t c_V\in L^2(\dom) 
	\quad\mbox{for } t\in(0,T)
$$
as an elliptic equation with 
bounded diffusion coefficient $D_V(f)$ and right-hand side in $L^p(\dom)$
with $p>d/2$. By \cite[Prop.~3.6]{Nit11}, there exists $\delta>0$
such that $c_V(t)\in C^\delta(\dom)$ and
$$
  \|c_V(t)\|_{C^\delta(\dom)}\le C\big(\|c_V\|_{L^2(\dom)} 
	+ \|\pa_t c_V\|_{L^p(\dom)}\big).
$$
The result follows by observing that $d\le 3$ implies that $p<2$.
The dependency of $\delta$ and $C_\delta$ on the data follows from
\cite[Theorem 8.24]{GiTr98}, which is the essential result needed in
the proof of \cite[Prop.~3.6]{Nit11}.
The regularity for the other concentrations is proved in a similar way.
\end{proof}

\begin{lemma}\label{lem.1+delta}
Let $c$ be a weak
solution to \eqref{1.c}--\eqref{1.cbic}. Furthermore, let 
$c_j^0\in C^{1+\delta}(\overline\dom)$ such that $\na c_j^0\cdot\nu=0$ on $\pa\dom$,
$\alpha_j$, $\beta_j\in C^0(\overline{\dom}\times[0,T])$, and
$f^0_j\in C^\delta(\overline\dom)$ for $j=V,D,M,U$, 
where $\delta>0$ is as in Lemma \ref{lem.hoelder}.
Then $c\in C^{1+\delta,(1+\delta)/2}(\overline{\dom}\times[0,T])$ and there
exists $C_{1+\delta}>0$ such that
$$
  \|c\|_{C^{1+\delta,(1+\delta)/2}(\overline{\dom}\times[0,T])}
	\le C_{1+\delta}.
$$
\end{lemma}

The space $C^{1+\delta,(1+\delta)/2}(\overline{\dom}\times[0,T])$ consist of
all functions being $C^{1+\delta}$ in space and $C^{(1+\delta)/2}$ in time;
see the Appendix for a precise definition.

\begin{proof}
We know from Lemma \ref{lem.hoelder} that $c(t)$ is H\"older continuous 
in $\overline\dom$ for a.e.\ $t\in(0,T)$. We claim that $f$ is
H\"older continuous in $\overline\dom\times[0,T]$.
Let $x,y\in\overline\dom$ and $\tau,t\in[0,T]$. 
We assume without loss of generality that
$\tau<t$. The Lipschitz continuity of $z\mapsto\exp(-z)$ implies, using the
explicit formula for $f_B$, that
\begin{align*}
  |f_B&(x,t)-f_B(y,t)| \le |f_B^0(x)-f_B^0(y)| + s_B\int_0^t|c_M(x,s)-c_M(y,s)|\dd s, \\
	&\le \|f_B^0\|_{C^\delta(\overline\dom)} |x-y|^\delta
	+ s_BC_\delta\big(\|c_M\|_{L^1(0,t;L^2(\dom))}
	+ \|\pa_t c_M\|_{L^1(0,t;L^2(\dom))}\big)|x-y|^\delta, \\
	|f_B&(x,t)-f_B(x,\tau)| \le |f_B^0(x)|s_B\int_\tau^t|c_M(x,s)|\dd s \\
	&\le \|f_B^0\|_{C^\delta(\overline\dom)}s_B\|c_M\|_{L^\infty(Q_T)}
	T^{1-\delta/2}|t-\tau|^{\delta/2},
\end{align*}
where we also used Lemma \ref{lem.hoelder}.
Similar estimates hold for $f_E$ and $f_F$. 
Thus, the assumptions of Theorem \ref{thm.hoelder} in the Appendix 
are fulfilled, yielding the statement.
\end{proof}

For the solvability of the SDEs \eqref{1.X}, we
need the $W^{2,\infty}(\dom)$ regularity of $c$. 
For this, we introduce the following space:
\begin{equation}\label{4.W21q}
  W^{2,1,q}(Q_T) := \big\{u\in L^q(0,T;W^{2,q}(\dom)):
	\pa_t u\in L^q(Q_T)\big\},
\end{equation}
where $q\in[1,\infty]$, equipped with the norm
$$
  \|u\|_{W^{2,1,q}} = \big(\|u\|^q_{L^q(0,T;W^{2,q}(\dom))} 
	+ \|\pa_t u\|^q_{L^q(Q_T)}\big)^{1/q}.
$$

\begin{lemma}\label{lem.W2infty}
Let $c$ be a weak solution
to \eqref{1.c}--\eqref{1.cbic}. Furthermore, let $c^0\in W^{2,\infty}(\dom)$.
Then $c\in W^{2,1,\infty}(\dom\times[0,T])$ and there exists a constant $C>0$
such that
$$
  \|c\|_{W^{2,1,\infty}} \le C.
$$
\end{lemma}

\begin{proof}
We show the statement for $c_V$, as the proofs for the other components are similar.
The function $u:=c_V-c_V^0$ solves the linear problem
\begin{align*}
  & \pa_t u - \diver(D_V(f)\na u) = g(x,t)	\quad\mbox{in }\dom,\ t>0, \\
	& \na u\cdot\nu = 0 \quad\mbox{on }\pa\dom, \quad u(0)=0\quad\mbox{in }\dom,
\end{align*}
where $g(x,t) := -\alpha_V(x,t) c_V(x,t) + \diver(D_V(f(x,t))\na c^0(x))$ is
bounded by a constant for $(x,t)\in\dom\times(0,T)$. This constant depends
on $\Lambda$, $c^0$, $f^0$, and $T$ but not on $c$. By Lemma \ref{lem.1+delta},
the diffusion coefficient $D_V(f)$ is an element of 
$C^{1+\delta,(1+\delta)/2}(\overline{\dom}\times[0,T])$. 
Therefore, an application of Theorem \ref{thm.strong} finishes the proof.
\end{proof}

For smooth initial data, we can obtain classical solutions to 
\eqref{1.c}--\eqref{1.cbic}. The following result is not needed for the
solvability of \eqref{1.X} but stated for completeness. 

\begin{lemma}
Let $c$ be a weak solution
to \eqref{1.c}--\eqref{1.cbic}. Furthermore, let
$\pa\dom\in C^{1+\delta}$, $\alpha,\beta\in C^{\delta,\delta/2}
(\overline\dom\times[0,T])$, $f^0\in C^{1+\delta}(\overline\dom)$,
and $c^0\in C^{2+\delta}(\overline\dom)$, where $\delta>0$ is as in 
Lemma \ref{lem.hoelder}. Furthermore, let $c$ be a weak solution to
\eqref{1.c}--\eqref{1.cbic}. Then 
$c\in C^{2+\delta,1+\delta/2}(\overline\dom\times[0,T])$. 
\end{lemma}

\begin{proof}
The result follows from Theorem \ref{thm.class} in the Appendix 
after bringing \eqref{1.c} in nondivergence form.
\end{proof}

\subsection{Stability and uniqueness}

The stability results are used for the solution of the SDEs; they also imply
the uniqueness of weak solutions. We start with a stability estimate
in the norms $L^\infty(0,T;L^2(\dom))$ and $L^2(0,T;H^1(\dom))$.
Let Assumptions (A1)--(A5) hold.

\begin{lemma}\label{lem.stab1}
Let $c_i$ for $i=1,2$ be weak solutions to \eqref{1.c}--\eqref{1.cbic} with
the same initial data $(c^0,f^0)$ but possibly different coefficients
$\alpha_i$ and $\beta_i$. 
Then there exists $C>0$, which is independent of $c_i$, such that for all $t\in[0,T]$,
\begin{align*}
  & \|(c_1-c_2)(t)\|_{L^2(\dom)} + \|c_1-c_2\|_{L^2(0,t;H^1(\dom))}
	\le h(t), \quad\mbox{where} \\
	& h(t) := C\big(\|\alpha_1-\alpha_{2}\|_{L^2(Q_T)}
	+ \|\beta_1-\beta_2\|_{L^2(Q_T)}\big).
\end{align*}
\end{lemma}

\begin{proof}
We first consider $c_V$. We take the difference of the equations satisfied by
$c_{1,V}-c_{2,V}$ and take the test function $c_{1,V}-c_{2,V}$ in its weak formulation.
This leads to
\begin{align*}
  \frac12\frac{\dd}{\dd t}&\int_\dom(c_{1,V}-c_{2,V})^2\dd x
	+ \int_\dom D_V(f_1)|\na(c_{1,V}-c_{2,V})|^2\dd x 
	+ \int_\dom\alpha_{1,V}(c_{1,V}-c_{2,V})^2\dd x \\
	&= \int_\dom(D_V(f_1)-D_V(f_2))\na c_{2,V}\cdot\na(c_{1,V}-c_{2,V})\dd x \\
	&\phantom{xx}{}+ \int_\dom (\alpha_{1,V}-\alpha_{2,V})c_{2,V}(c_{1,V}-c_{2,V})\dd x.
\end{align*}
Let $\eps:=\min\{D_j^i:j=V,D,M,U,\,i=B,E,F\}>0$.
Using Young's inequality and the estimate $\|(f_{1}-f_{2})(t)\|_{L^2(\dom)}
\le C\|c_{1}-c_{2}\|_{L^1(0,T;L^2(\dom))}$ from Lemma \ref{lem.ODEdiff}, we find that
\begin{align*}
  \frac{\dd}{\dd t}&\|c_{1,V}-c_{2,V}\|_{L^2(\dom)}^2
	+ \frac{\eps}{2}\|\na(c_{1,V}-c_{2,V})\|_{L^2(\dom)}^2 \\
	&\le C(\eps)\|\na c_{2,V}\|_{L^\infty(Q_T)}^2\|D_V(f_1)-D_V(f_2)\|_{L^2(\dom)}^2 \\
	&\phantom{xx}{}+ \|c_{2,V}\|_{L^2(Q_T)}^2\|\alpha_{1,V}-\alpha_{2,V}\|_{L^2(\dom)}^2 
	+ \|c_{1,V}-c_{2,V}\|_{L^2(\dom)}^2 \\
	&\le C\|c_{1}-c_{2}\|_{L^2(\dom)}^2 
	+ C\|\alpha_{1,V}-\alpha_{2,V}\|_{L^2(\dom)}^2.
\end{align*}
The estimates for $c_{1,j}-c_{2,j}$ with $j=D,M,V$ are similar. This gives
$$
  \frac{\dd}{\dd t}\|c_1-c_2\|_{L^2(\dom)}^2 
	+ \frac{\eps}{2}\|\na(c_1-c_2)\|_{L^2(\dom)}^2
	\le Ch(t) + C\|c_1-c_2\|_{L^2(\dom)}^2.
$$
An application of Gronwall's lemma finishes the proof.
\end{proof}

A stability estimate can also be proved with respect to the $H^2(\dom)$ norm.

\begin{lemma}\label{lem.stab2}
Let $c_i$ for $i=1,2$ be weak solutions to \eqref{1.c}--\eqref{1.cbic} with
the same initial data $(c^0,f^0)$ but possibly different coefficients
$\alpha_i$ and $\beta_i$. Then there exists $C>0$ such that for all
$t\in[0,T]$,
\begin{align*}
  \|\pa_t&(c_1-c_2)\|_{L^2(Q_T)} + \|c_1-c_2\|_{L^\infty(0,T;H^1(\dom))}
	+ \|c_1-c_2\|_{L^2(0,T;H^2(\dom))} \\
	&\le C\big(\|\alpha_1-\alpha_2\|_{L^2(Q_T)} + \|\beta_1-\beta_2\|_{L^2(Q_T)}\big),
\end{align*}
where $C>0$ depends on $\|c_1\|_{L^\infty(0,T;W^{2,\infty}(\dom))}$.
\end{lemma}

\begin{proof}
The difference $u:=c_{1,V}-c_{2,V}$ is the solution to the linear problem
\begin{equation}\label{4.u}
\begin{aligned}
  & \pa_t u - \diver(D_V(f_1)\na u) = g(x,t)\quad\mbox{in }\dom,\ t>0, \\
	& \na u\cdot\nu = 0\quad\mbox{on }\pa\dom, \quad u(0)=0\quad\mbox{in }\dom,
\end{aligned}
\end{equation}
where, by Lemma \ref{lem.W2infty}, the right-hand side
\begin{equation}\label{4.g}
  g := -\diver\big((D_V(f_1)-D_V(f_2))\na c_{2,V}\big)
	+ \alpha_{1,V}(c_{1,V}-c_{2,V}) + (\alpha_{1,V}-\alpha_{2,V})c_{2,V}
\end{equation}
is an element of $L^2(Q_T)$. Since the diffusion coefficient is bounded, we can apply
Theorem \ref{thm.pat} in the Appendix to conclude that 
$$
  \|u\|_{L^\infty(0,T;H^1(\dom))} + \|\pa_t u\|_{L^2(Q_T)}
	\le C\|g\|_{L^2(Q_T)}.
$$
For the estimate of the right-hand side, we recall from Lemma \ref{lem.ODEdiff} that
$$
  \|\na(f_1-f_2)\|_{L^2(Q_T)} \le C\|c_1-c_2\|_{L^1(0,T;H^1(\dom))}.
$$
Then, using the linearity of $D_V$ and the estimate for $c_{2,V}$ from Lemma 
\ref{lem.W2infty}, we infer that
\begin{align*}
  \|g\|_{L^2(Q_T)} &\le C\|\na(f_1-f_2)\|_{L^2(Q_T)}\|\na c_{2,V}\|_{L^\infty(Q_T)}
	+ C\|f_1-f_2\|_{L^2(Q_T)}\|\Delta c_{2,V}\|_{L^\infty(Q_T)} \\
	&\phantom{xx}{}+ \|\alpha_{1,V}\|_{L^\infty(Q_T)}\|u\|_{L^2(Q_T)}
	+ \|\alpha_{1,V}-\alpha_{2,V}\|_{L^2(Q_T)}\|c_{2,V}\|_{L^\infty(Q_T)} \\
	&\le C\big(\|c_1-c_2\|_{L^2(0,T;H^1(\dom))} 
	+ \|\alpha_{1,V}-\alpha_{2,V}\|_{L^2(Q_T)}\big).
\end{align*}
The difference $c_1-c_2$ in the $L^2(0,T;H^1(\dom))$ norm can be estimated
according to Lemma \ref{lem.stab1}. Therefore,
\begin{equation}\label{4.uH1}
  \|u\|_{L^\infty(0,T;H^1(\dom))} + \|\pa_t u\|_{L^2(Q_T)}
	\le C\big(\|\alpha_1-\alpha_2\|_{L^2(Q_T)} + \|\beta_1-\beta_2\|_{L^2(Q_T)}\big).
\end{equation}
Similar estimates can be derived for the differences $c_{1,j}-c_{2,j}$
($j=D,M,U$). 

To estimate $u$ in the $L^2(0,T;H^2(\dom))$ norm, we use the inequality
$$
  \|u\|_{H^2(\dom)} \le C\big(\|\Delta u\|_{L^2(\dom)} + \|u\|_{L^2(\dom)}\big).
$$
Thus, it remains to consider $\Delta u$. We deduce from
\begin{align*}
  D_{V}(f_1)\Delta u &= \diver\big(D_{V}(f_1)\na c_{1,V}-D_{V}(f_2)\na c_{2,V}\big)
	- \na(D_{V}(f_1)-D_{V}(f_2))\cdot\na c_{2,V} \\
	&\phantom{xx}{}- (D_{V}(f_1)-D_{V}(f_2))\Delta c_{2,V}
	- \na D_{V}(f_1)\cdot\na(c_{1,V}-c_{2,V}) \\
	&= \pa_t u - \alpha_{1,V}u - (\alpha_{1,V}-\alpha_{2,V})c_{2,V}
	- \na(D_{V}(f_1)-D_{V}(f_2))\cdot\na c_{2,V} \\
	&\phantom{xx}{}- (D_{V}(f_1)-D_{V}(f_2))\Delta c_{2,V} - \na D_{V}(f_1)\cdot\na u
\end{align*}
and $\|\na(D_{V}(f_1)-D_{V}(f_2))\|_{L^2(Q_T)}\le C\|c_1-c_2\|_{L^2(0,T;H^1(\dom))}$ 
(see Lemma \ref{lem.ODEdiff}) that
$$
  \|\Delta u\|_{L^2(Q_T)} 
	\le C\big(\|\alpha_1-\alpha_2\|_{L^2(Q_T)} + \|\beta_1-\beta_2\|_{L^2(Q_T)}
	+ \|u\|_{L^2(0,T;H^1(\dom))}\big).
$$
We infer from \eqref{4.uH1} and related inequalities for $c_{1,j}-c_{2,j}$ that
$$
  \|\Delta(c_{1}-c_{2})\|_{L^2(Q_T)} 
	\le C\big(\|\alpha_1-\alpha_2\|_{L^2(Q_T)} + \|\beta_1-\beta_2\|_{L^2(Q_T)}\big),
$$
which concludes the proof.
\end{proof}

\begin{lemma}\label{lem.stabW214}
Let $c_i$ for $i=1,2$ be weak solutions to \eqref{1.c}--\eqref{1.cbic} with
the same initial data $(f^0,c^0)$ but possibly different coefficients
$\alpha_i$ and $\beta_i$. Then there exists $C>0$ such that for all
$t\in[0,T]$,
$$
  \|c_1-c_2\|_{W^{2,1,4}}
	\le C\big(\|\alpha_1-\alpha_2\|_{L^4(Q_T)} + \|\beta_1-\beta_2\|_{L^4(Q_T)}\big),
$$
recalling definition \eqref{4.W21q} of $W^{2,1,4}(Q_T)$.
\end{lemma}

\begin{proof}
Let $u=c_{1,V}-c_{2,V}$ be the solution to \eqref{4.u}.
Since $g\in L^4(Q_T)$ by Lemma \ref{lem.W2infty} (recall definition \eqref{4.g}
of $g$), Theorem \ref{thm.strong} in the Appendix
shows that $u\in W^{2,1,4}(Q_T)$ and, because of $\na c_{2,V}\in L^\infty(Q_T)$,
\begin{align*}
  \|c_{1,V}-c_{2,V}\|_{W^{2,1,4}} &\le C\|g\|_{L^4(Q_T)}
	\le C\big(\|c_{1}-c_{2}\|_{L^1(0,T;W^{1,4}(\dom))} 
	+ \|c_{1,V}-c_{2,V}\|_{L^4(Q_T)} \\
	&\phantom{xx}{}+ \|\alpha_{1,V}-\alpha_{2,V}\|_{L^4(Q_T)}\big).
\end{align*}
The first term on the right-hand side can be estimated by using the embedding
$H^{2}(\dom)\hookrightarrow W^{1,4}(\dom)$ and Lemma \ref{lem.stab2}:
\begin{align*}
  \|c_{1}-c_{2}\|_{L^1(0,T;W^{1,4}(\dom))} 
	&\le C\|c_{1}-c_{2}\|_{L^2(0,T;H^2(\dom))} \\
	&\le C\big(\|\alpha_1-\alpha_2\|_{L^2(Q_T)} + \|\beta_1-\beta_2\|_{L^2(Q_T)}\big).
\end{align*}
We deduce from the embedding $H^1(\dom)\hookrightarrow L^4(\dom)$ and
Lemma \ref{lem.stab2} again that
\begin{align*}
  \|c_{1,V}-c_{2,V}\|_{L^4(Q_T)} 
	&\le C\|c_{1,V}-c_{2,V}\|_{L^4(0,T;H^1(\dom))} \\
	&\le C\big(\|\alpha_1-\alpha_2\|_{L^2(Q_T)} + \|\beta_1-\beta_2\|_{L^2(Q_T)}\big).
\end{align*}
This gives
$$
  \|c_{1,V}-c_{2,V}\|_{W^{2,1,4}} 
	\le C\big(\|\alpha_1-\alpha_2\|_{L^4(Q_T)} + \|\beta_1-\beta_2\|_{L^2(Q_T)}\big).
$$
The estimates for $c_{1,j}-c_{2,j}$ ($j=D,M,U$) are similar.
\end{proof}


\section{Solution of the stochastic differential equations}\label{sec.X}

Let $\alpha$, $\beta$ be given by \eqref{1.ab}. We first study the measurability
of $(c,f)$.

\begin{lemma}\label{lem.meas}
Let $f^0\in L^\infty(\Omega;C^{1+\delta}(\overline\dom))$ and
$c^0\in L^\infty(\Omega;W^{2,\infty}(\dom))$ such that $\na c^0_j\cdot\nu=0$
on $\pa\dom$, $j=V,D,M,U$. Furthermore, let $(c,f)$ be a pathwise solution
to \eqref{1.f}--\eqref{1.cbic}. Then $f$, $\na f$ are measurable as maps from
$(\Omega\times\dom\times[0,t],\F_t\times\B(\dom)\times\B([0,t]))$ to
$\B(\R^3)$, and $c$, $\na c$ are measurable as maps from
$(\Omega\times\dom\times[0,t],\F_t\times\B(\dom)\times\B([0,t]))$ 
to $\B(\R^4)$ for all $t\in [0,T]$.
In particular, these functions are progressively measurable.
\end{lemma}

\begin{proof}
Since $f_j$ can be represented as a function depending on the time integral of $c$,
it is sufficient to show the measurability of $c_j$. The continuity of
the potentials defining $\alpha_j$ and $\beta_j$ in \eqref{1.ab} shows that 
$\alpha_j$ and $\beta_j$ are processes with c\`adl\`ag paths almost surely. 
By approximating the initial data $c^0$, $f^0$
and the processes $\alpha_j$, $\beta_j$ by suitable simple processes, 
which are adapted to the filtration by construction,
we can obtain the $\F_t$-measurability of $c_j(t):\Omega\to C^1(\overline\dom)$ 
for $t\in[0,T]$. We conclude from Lemma \ref{lem.stab1} and the compactness of
$W^{2,\infty}(\dom)\subset C^{1+\delta}(\overline\dom)$ in $C^1(\overline\dom)$
the measurability of $c_j$ as the limit of measurable functions.
For details of this construction, we refer to \cite[Section 3.3]{HSZS18}.
It is known that c\`adl\`ag processes, which are adapted to the filtration, are also
progressively measurable \cite[Prop.~1.13]{KaSh91}. If the filtration is complete,
this holds also true for processes having c\`adl\`ag paths almost surely.
The estimate $\|c_j(t)-c_j(s)\|_{C^1(\overline\dom)}\le C|t-s|^\delta$,
which follows from Lemma \ref{lem.hoelder}, implies that $c_j(t)$ has
almost surely continuous paths and consequently, $c_j(t)$ is progressively
measurable. To be precise, this yields the measurability of $c_j$ as a function
from $(\Omega\times[0,t],\F_t\times\B([0,t]))$ to 
$(C^1(\overline\dom),\B(C^1(\overline\dom)))$ for every $t\in[0,T]$.

The function $(c,x)\mapsto c(x)$, $C^1(\overline\dom)\times\overline\dom\to\R^4$,
is continuous and hence, it is measurable as a mapping from 
$(C^1(\overline\dom),\B(C^1(\overline\dom)))$ to $(\R^4,\B(\R^4))$. 
Now, we can write $c(\omega,x,t)$ as the concatenation 
$$
  (\omega,x,t)\mapsto (c(\omega,x,t),x)\mapsto c(\omega,x,t),\quad
  \Omega\times\overline\dom\times[0,T]\to C^1(\overline\dom)\times
	\overline\dom\to\R^4, 
$$
of measurable functions, which yields the measurability of $c$.
In a similar way, we can prove the measurability of $\pa c_j/\pa x_i$
for $i=1,2,3$ by considering the continuous mapping $(c,x)\mapsto 
(\pa c/\pa x_i)(x)$.
\end{proof}

\begin{lemma}\label{lem.exX}
Let Assumptions (A1)--(A5) hold.
Then there exists a unique, progressively measurable solution $(X_i^k)$ to 
\eqref{1.X} such that $X_i^k(t)\in\overline\dom$ a.s.\ for every $t\in[0,T]$.
\end{lemma}

\begin{proof}
We extend the coefficients $g_i$ and $\sigma_i$ by setting them to zero outside of
$\dom$. The extended coefficients are still uniformly Lipschitz continuous.
We infer from Lemma \ref{lem.meas} that $g_i$ is progressively measurable.
Thus, by \cite[Theorem 32.3]{Kal21}, there exists a strong solution to \eqref{1.X}.

It remains to show that $X_i^k(t)\in\overline\dom$ a.s. 
Let $\phi$ be a smooth test function 
satisfying $\operatorname{supp}\phi\subset\dom^c$.
We obtain from It\^o's lemma that
\begin{equation}\label{4.dphi}
  \dd\phi(X_i^k) = \na\phi(X_i^k)\cdot g_i[c,f](X_i^k,t)\dd t
	+ \frac12\sigma_i(X_i^k)^2\Delta\phi(X_i^k)\dd t 
	+ \na\phi(X_i^k)\cdot\sigma_i(X_i^k)\dd W_i^k.
\end{equation}
If $X_i^k(t)\in\dom$, we have $\phi(X_i^k)=0$. If $X_i^k(t)\in\dom^c$ then
$g_i[c,f](X_i^k(t),t)=0$ by Assumption (A3) and $\sigma_i(X_i^k(t))=0$ 
by Assumption (A2).
Equation \eqref{4.dphi} then shows that $\phi(X_i^k(t))=\phi(X_i^0)=0$ and
$X_i^k(t)\in\overline{(\operatorname{supp}\phi)^c}$ a.s. Since $\phi$ with
$\operatorname{supp}\phi\subset\dom^c$ was arbitrary, we conclude that
$X_i^k(t)\in\overline\dom$ a.s.\ for $t\in(0,T)$.
\end{proof}


\section{Proof of Theorem \ref{thm.ex}}\label{sec.ex}

The fixed-point operator is defined as a function that maps 
$\widetilde{X}\mapsto(\alpha,\beta_D)\mapsto (c,f)\mapsto X$, where
$(\alpha,\beta_D)$ are defined in \eqref{1.ab} with $X$ replaced by $\widetilde{X}$.
To define its domain, we need some preparations.
For given $R>0$, we introduce the following space:
\begin{align}\label{def.YR}
  Y_R(0,T;\dom) &:=\big\{X\in C^{1/2}([0,T];L^4(\Omega)):
	\|X\|_{C^{1/2}([0,T];L^4(\Omega))}\le R,\\ 
	&\phantom{xxm} X(t)\mbox{ is }\F_t\mbox{-measurable},\
	X(t)\in\overline\dom\mbox{ a.s. for all }t\in[0,T]\big\}, \nonumber
\end{align}
equipped with the standard norm of $C^{0}([0,T];L^4(\Omega))$.

\begin{lemma}
The space $Y_R(0,T;\dom)$ is complete. Furthermore, any $X\in Y_R(0,T;\dom)$
has a progressively measurable modification with almost surely H\"older 
continuous paths.
\end{lemma}

\begin{proof}
Let $(X_n)$ be a Cauchy sequence in $Y_R(0,T;\dom)$ and let $\eps>0$. Then
there exists $N\in\N$ such that for all $n,m\ge N$,
$$
  \|X_n(t)-X_m(t)\|_{L^4(\Omega)} \le \|X_n-X_m\|_{C^0([0,T];L^4(\Omega))} < \eps.
$$
For any $t\in[0,T]$, $(X_n(t))$ is a Cauchy sequence in $L^4(\Omega)$. Consequently,
$X_n(t)\to X(t)$ in $L^4(\Omega)$, where $X(t)\in L^4(\Omega)$ is $\F_t$-measurable.
Furthermore, there exists a subsequence of $(X_n(t))$ (not relabeled) that
converges pointwise to $X(t)$ a.s., proving that $X(t)\in\overline\dom$ a.s.
The definition of the H\"older norm implies that
$\|X_n(t)-X_n(s)\|_{L^4(\Omega)}\le R|t-s|^{1/2}$ for all $s,t\in[0,T]$.
This gives in the limit $n\to\infty$ that 
$\|X(t)-X(s)\|_{L^4(\Omega)}\le R|t-s|^{1/2}$ and consequently
$X\in C^{1/2}(0,T;\dom)$. We conclude that $X\in Y_R(0,T;\dom)$.
By the Kolmogorov continuity criterium, (a modification of) $X$ has almost 
surely H\"older continuous paths. As $X(t)$ is an adapted process with respect 
to the filtration $\F_t$, $X$ is progressively measurable.
\end{proof}

\begin{lemma}
Let $\widetilde{X}\in Y_R(0,T;\dom)$ for some $R>0$,
and let $(c,f)$ be a solution to \eqref{1.f}--\eqref{1.cbic}, where
$\alpha$, $\beta$ are given by \eqref{1.ab} with $X$ replaced by
$\widetilde{X}$. Then there exists $R_0>0$ not depending
on $R$ such that the solution $X$ to \eqref{1.X} satisfies $X\in Y_{R_0}(0,T;\dom)$.
\end{lemma}

\begin{proof}
According to Lemma \ref{lem.W2infty},
$c$ is bounded in the $L^\infty(0,T;W^{2,\infty}(\dom))$ norm
by a constant that is independent of $R$. Then, by Lemma \ref{lem.exX},
there exists a unique solution $X$ to \eqref{1.f}. Since
$X_i^k(t)\in\overline\dom$ a.s., $c$, $\na c$, $f$, $\na f$ 
are bounded uniformly in $R$,
i.e., there exists $K=K(c^0,f^0)>0$, which is independent of $R$, such that
$|g_i[c,f](X_i^k(t),t)|\le K$ a.s. Thus, for $s,t\in[0,T]$,
using the Burkholder--Davis--Gundy inequality,
\begin{align*}
  \E|X_i^k(t)-X_i^k(s)|^4 
	&\le C(K)|t-s|^4 + C\E\bigg(\int_s^t\sigma(X_i^k(s))\dd W_i^k(s)\bigg)^4 \\
	&\le C(K)|t-s|^4 + C\E\bigg(\int_s^t\sigma(X_i^k(s))^2\dd s\bigg)^2 \\
	&\le C(K)\big(|t-s|^2 + \|\sigma\|_{L^\infty(\dom)}\big)|t-s|^2
	\le C(K,T,\sigma,\dom)|t-s|^2.
\end{align*}
The lemma follows after choosing $R_0:=C(K,T,\sigma,\dom)^{1/4}$.
\end{proof}

The previous lemma shows that the fixed-point operator 
$\Phi:Y_{R_0}(0,T;\dom)\to Y_{R_0}(0,T;\dom)$, $\widetilde{X}\mapsto X$, 
is well defined. We need to verify that $\Phi$ is a contraction. 
We first prove an auxiliary result.

\begin{lemma}\label{lem.diffX}
Let $(c,f)$ and $(c',f')$ be progressively measurable solutions
to \eqref{1.f}--\eqref{1.cbic}, where $\alpha$, $\beta$ are given by \eqref{1.ab}
with $X$ replaced by $\widetilde{X}$, $\widetilde{X}'\in Y_R(0,T;\dom)$ 
for some $R>0$, respectively.
Then the associated solutions $X$ and $X'$ to \eqref{1.X} satisfy
$$
  \E|X(t)-X'(t)|^4 \le Ct\int_0^t\E\|c(s)-c'(s)\|_{C^1(\overline\dom)}^4\dd s,
$$
where the constant $C>0$ does not depend on $R$, $(c,f)$, or $(c',f')$. 
\end{lemma}

\begin{proof}
The It\^o integral representation of $X(t)-X'(t)$ gives
\begin{align}\label{4.est1}
  \E|X_i^k&(t)-(X')_i^k(t)|^4 
	\le C\E\bigg(\int_0^t\big(g_i[c,f](X_i^k(s),s)-g_i[c',f']((X')_i^k(s),s)\big)\dd s
	\bigg)^4 \\
	&\phantom{xx}{}
	+ C\E\bigg(\int_0^t\big(\sigma(X_i^k(s))-\sigma((X')_i^k(s))\big)\dd W_i^k(s)\bigg)^4 
	=: I_1 + I_2. \nonumber
\end{align}
It follows from Assumption (A4) that
\begin{align*}
  I_1 &\le C\E\bigg|\int_0^t\big(g_i[c,f](X_i^k(s),s)-g_i[c',f'](X_i^k(s),s)\big)\dd s
	\bigg|^4 \\
	&\phantom{xx}{}
	+ \E\bigg|\int_0^t\big(g_i[c',f'](X_i^k(s),s)-g_i[c',f']((X')_i^k(s),s)\big)\dd s
	\bigg|^4 \\
	&\le L_1^4\E\big(1+\|c\|_{L^\infty(0,T;C^1(\overline\dom))}\big)^4
	\bigg(\int_0^t\|c(s)-c'(s)\|_{C^1(\overline\dom)}\dd s\bigg)^4 \\
	&\phantom{xx}{}
	+ L_2^4\E\big(1+\|c'\|_{L^\infty(0,T;W^{2,\infty}(\dom))}\big)^4
	\bigg(\int_0^t|X(s)-X'(s)|\dd s\bigg)^4.
\end{align*}
Furthermore, by the Burkholder--Davis--Gundy inequality and the Lipschitz
continuity of $\sigma$,
$$
  I_2 \le C	\E\bigg(\int_0^t\big((\sigma(X_i^k(s))-\sigma((X')_i^k(s))\big)^2\dd s
	\bigg)^2 \le C\E\bigg(\int_0^t|X(s)-X'(s)|^2\dd s\bigg)^2.
$$
We insert these estimates into \eqref{4.est1} and use H\"older's inequality:
\begin{align*}
  \E|X(t)-X'(t)|^4 
	&\le Ct^3\E\int_0^t\|c(s)-c'(s)\|_{C^1(\overline\dom)}^4\dd s \\
	&\phantom{xx}{}+ Ct^3\E\int_0^t|X(s)-X'(s)|^4\dd s
	+ Ct\E\int_0^t|X(s)-X'(s)|^4\dd s.
\end{align*}
Then Gronwall's lemma concludes the proof.
\end{proof}

We prove now that $\Phi:Y_{R_0}(0,T;\dom)\to Y_{R_0}(0,T;\dom)$, 
$\widetilde{X}\mapsto X$, is a contraction.
By Lemmas \ref{lem.diffX} and \ref{lem.stabW214}, we have
\begin{align*}
  \E|\Phi(X(t))-\Phi(X'(t))|^4 &\le Ct\int_0^t\E\|c(s)-c'(s)\|_{C^1(\overline\dom)}^4
	\dd s \\
	&\le Ct\E\big(\|\alpha-\alpha'\|^4_{L^4(Q_T)} + \|\beta-\beta'\|^4_{L^4(Q_T)}\big) \\
	&\le Ct\E\|X-X'\|_{L^4(0,t,L^4(\dom))}^4 = Ct\int_0^t\E|X(s)-X'(s)|^4\dd s.
\end{align*}
We iterate this inequality to find after $n$ times that
\begin{align*}
  \E|\Phi^n(X(t))-\Phi^n(X'(t))|^4 &\le (Ct)^n\int_0^t\int_0^{s_1}
	\cdots\int_0^{s_{n-1}}\E|X(s_n)-X'(s_n)|^4\dd s_n\cdots\dd s_1 \\
	&\le (Ct)^n\frac{t^n}{n!}\sup_{0<s<t}\E|X(s)-X'(s)|^4.
\end{align*}
We conclude that
$$
  \sup_{0<s<T}\big(\E|\Phi^n(X(t))-\Phi^n(X'(t))|^4\big)^{1/4}
	\le \frac{(CT^2)^{n/4}}{(n!)^{1/4}}\sup_{0<s<T}\big(\E|X(s)-X'(s)|^4\big)^{1/4}.
$$
The sequence $(CT^2)^{n/4}/(n!)^{1/4}$ converges to zero as $n\to\infty$.
Hence, there exists $n\in\N$ such that $\Phi^n$ is a contraction.
By the variant \cite[Theorem 2.4]{Lat14} of Banach's fixed-point theorem,
$\Phi$ has a fixed point, proving Theorem \ref{thm.ex}.


\section{Numerical experiments}\label{sec.num}

We illustrate the dynamics of the tip and stalk cells in the two-dimensional
ball $\dom=B_R(0)$ around the origin with radius $R=500$ (in units of $\mu$m).
Let $h=10$ be the space step size and introduce the grid points
$x_{ij}=((k-i)h,(k-j)h)\in\R^2$, where $i,j=0,\ldots,2k$ and $k=R/h$.
The time step size equals $\tau=1$ (in units of seconds). 

The stochastic differential equations \eqref{1.X} are discretized by using
the Euler--Maruyama scheme. The nonlinearity $g_i[c,f]$ is chosen as in
\eqref{1.g} with $M$, $\gamma$, and $\lambda$ given in Appendix \ref{app.model}.
Furthermore, $\alpha_{0}$ and $z$ are taken as in 
\cite[formulas (10) and (14)]{VeGe12}.
Compared to \cite{BMGV16}, we neglect the contribution of the Hertz contact 
mechanics regarding $z$ to guarantee the boundary condition
 $g_i[c,f](\cdot,t)=0$ on
$\pa\dom$. We choose the continuous radially symmetric stochastic diffusion
$$
  \sigma(x) = \left\{\begin{array}{ll}
	0 &\quad\mbox{for }|x|\ge R, \\
	(1/R)\sqrt{(R/10)^2 - [R/10-(R-|x|)^2]}
	&\quad\mbox{for }9R/10<|x|<R, \\
	1/10 &\quad\mbox{for }|x|\le 9R/10.
	\end{array}\right.
$$

The solutions \eqref{1.fBF} to the ordinary differential equations \eqref{1.f} 
are written iteratively as
$$
  f_B(x,(n+1)\tau) = f_{B}(x,n\tau)\exp\bigg(-s_B\int_0^\tau c_M(x,s+n\tau)\dd s\bigg),
	\quad n\in\N,
$$
and similarly for $f_F$. The integral is approximated by the trapezoid rule
$$
  \int_0^\tau c_M(x,s+n\tau)\dd s \approx \frac{\tau}{2}(c_{M,ij}^n+c_{M,ij}^{n+1}),
$$
where $c_{M,ij}^n$ approximates $c_M(x_{ij},n\tau)$. We set
$f_{ij}^{n}:=(f_B,f_E,f_F)(x_{ij},n\tau)$. 

Finally, we discretize the reaction-diffusion equations \eqref{1.c} using
the forward Euler method and the central finite-difference scheme
$$
  \diver(D_V(f)\na c_V) \approx \frac{1}{h}\big(J_{i+1/2,j} 
  - J_{i-1/2,j} + J_{i,j+1/2} - J_{i,j-1/2}\big),
$$
where
\begin{align*}
  J_{i+1/2,j} &= \frac{1}{2h}(D_V(f_{i+1,j}^n) + D_V(f_{ij}^n))
  (c_{i+1,j}^{n+1}-c_{ij}^{n+1}), \\
  J_{i,j+1/2} &= \frac{1}{2h}(D_V(f_{i,j+1}^n) + D_V(f_{ij}^n))
  (c_{i,j+1}^{n+1}-c_{ij}^{n+1}).
\end{align*}
Notice that we obtain a semi-implicit scheme. 
The resulting linear system of equations is implemented
in the Python-based software environment {\em SciPy}
using sparse matrices and solved by using the {\tt spsolve} function
from the {\tt scipy.sparse.linalg} package.

The potentials $V_j^k$, used in \eqref{1.c}, are given by
$$
  V_j^k(x) = \frac{1}{IR_m^2}\exp\bigg(-\frac{R_m^2}{R_m^2-|x|^2}\bigg),
	\quad x\in\dom,\ j=D,M,U,V,
$$
where $R_m=12.5$, and $I>0$ is a normalization constant to ensure that
$\int_{\R^2}V_j^k(x)\dd x=1$.

It remains to define the initial conditions.  
The initial positions of the endothelial cells $X_i^{0,k}$ ($i=1,2$, $k=1,\ldots,N_i$)
are given by 
$$
  X_i^{0,k} = \begin{pmatrix} r\sin\phi \\ r\cos\phi \end{pmatrix},
$$
where $(r,\phi)$ is uniformly drawn from the set $[0.65R,0.75R]\times[0,\pi/2]$
and $R0.65=325$, $R_{c}=0.75R=375$. 
The initial volume fractions are
$$
  f_F^0(x) = \left\{\begin{array}{ll}
	0 &\quad\mbox{for }|x|\ge R_f, \\
	0.4(1 - \cos(\frac{\pi}{0.3R_f}(R_f-|x|))
	&\quad\mbox{for }0.7R_f<|x|<R_f, \\
	0.8 &\quad\mbox{for }|x|\le 0.7R_f,
	\end{array}\right.
$$
as well as $f_B^0=0.2f_F^0$ and $f_E^0=1-f_B^0-f_F^0$. We choose
the initial VEGF concentration
$$
  c_V^0(x) = 0.1\exp\bigg(-\frac{R_c}{\sqrt{R_c^2-|x|^2}}\bigg)
	\mathrm{1}_{B_{R_c}}(x),
$$
which is concentrated at the origin, and assume that the concentrations of the
remaining proteins vanish, $c_D^0=c_M^0=c_U^0=0$ in $\dom$, as they are
segregated by the tip cells.

We choose $N_1=2$ tip cells and $N_2=200$ stalk cells. 
Figure \ref{fig.cells} shows the positions of the tip and stalk cells at different
times. The tip cells segregate the DLL4 protein, and the stalk cells detect the
local increase of the DLL4 concentration, such that they follow the corresponding
tip cell. This effect is slightly more pronounced for the tip cell that starts in
an environment with a dense stalk cell population. The position of this tip cell
is closer to the origin than the other tip cell
with a higher VEGF concentration, 
leading to a relatively high production of DLL4 proteins.
The stalk cells, which do not follow a tip cell, are primarily
influenced by the stiffness gradient $\na(f_B+f_F)$
and the strain energy density $M$, which incorporates contact mechanics, 
resulting to a spreading of these cells.

\begin{figure}[ht]
\includegraphics[width=0.32\textwidth]{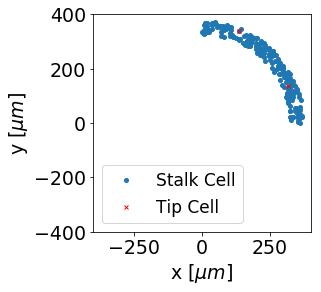}
\includegraphics[width=0.32\textwidth]{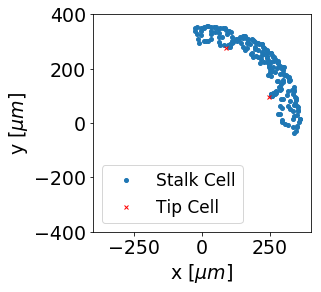}
\includegraphics[width=0.32\textwidth]{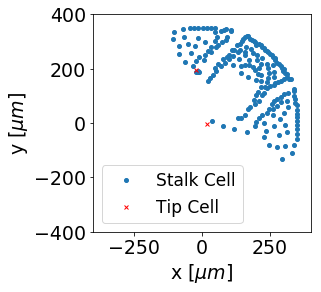}
\caption{Positions of two tip cells (red crosses) and 200 stalk cells (blue dots)
at times $T=0$\,s, $T=400$\,s, and $T=1600$\,s. }
\label{fig.cells}
\end{figure}

The protein concentrations are shown in Figure \ref{fig.conc}. As the diffusion
coefficient for VEGF is much larger than the reaction rate $s_V$, the concentration
of the VEGF protein becomes uniform in the large-time limit. The DLL4, MMP, and uPA
proteins are produced by the tip cells and hence follow their paths.
The corresponding concentrations increase with the availability of VEGF and decrease
due to consumption by the stalk cells or by getting exhausted from breaking down 
the fibrin matrix or the boundary membrane. Since the diffusion is slow,
the changes in the concentration are local up to time $T=1600$\,s.

\begin{figure}[ht]
\includegraphics[width=0.32\textwidth]{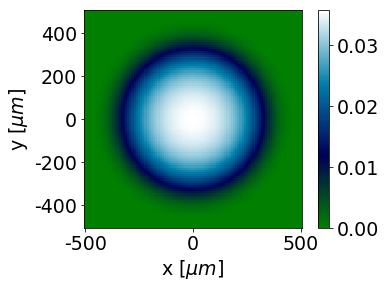} 
\includegraphics[width=0.32\textwidth]{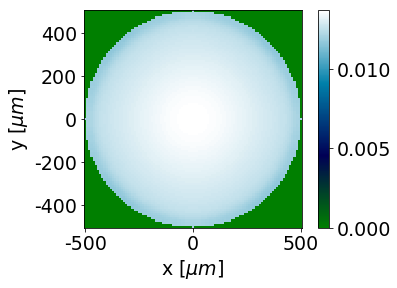}
\includegraphics[width=0.32\textwidth]{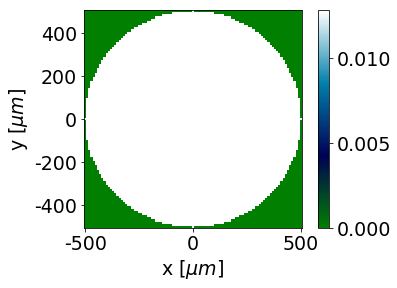}
\includegraphics[width=0.32\textwidth]{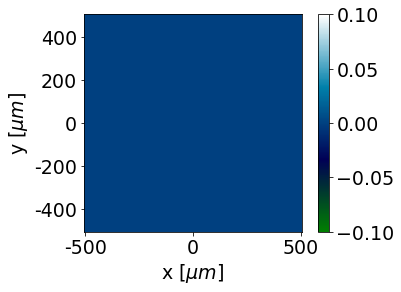} 
\includegraphics[width=0.32\textwidth]{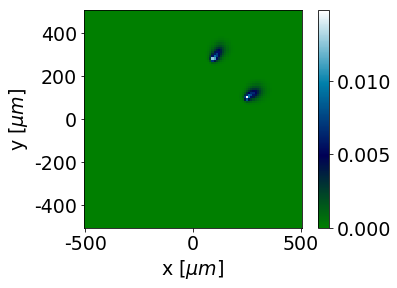}
\includegraphics[width=0.32\textwidth]{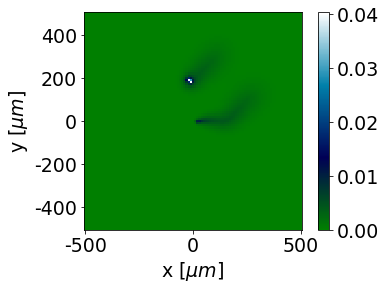}
\includegraphics[width=0.32\textwidth]{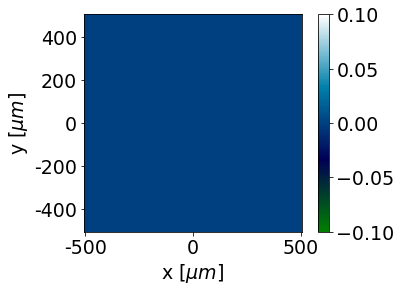} 
\includegraphics[width=0.32\textwidth]{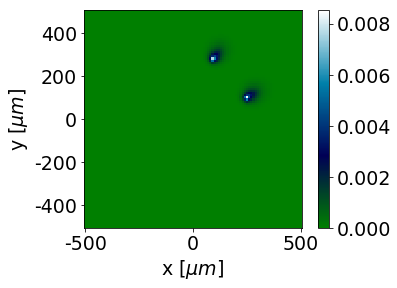}
\includegraphics[width=0.32\textwidth]{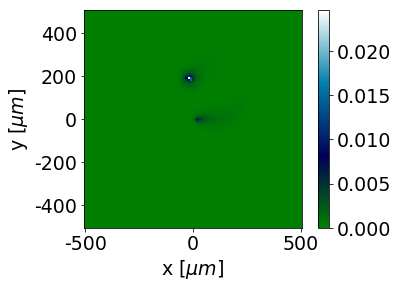}
\includegraphics[width=0.32\textwidth]{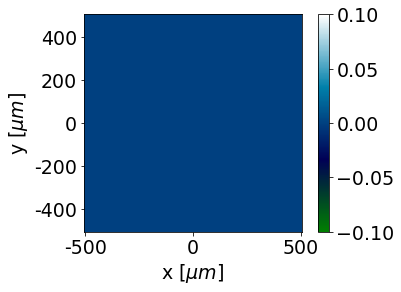} 
\includegraphics[width=0.32\textwidth]{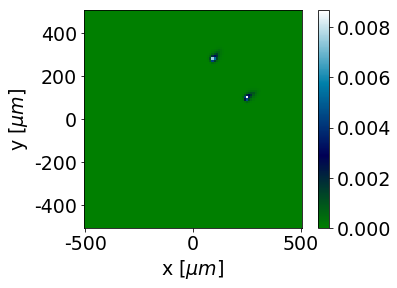}
\includegraphics[width=0.32\textwidth]{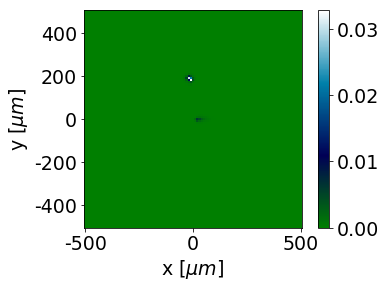}
\caption{Concentrations of the proteins VEGF (first row), DLL4 (second row),
MMP (third row), and uPA (last row)
at times $T=0$\,s (left column), $T=400$\,s (middle column), 
and $T=1600$\,s (right column).}
\label{fig.vf}
\end{figure}

We present the volume fractions of the basement membrane, fibrin matrix,
and extracellular fluid in Figure \ref{fig.vf}. 
The membrane and fibrin matrix are degraded by the
MMP and uPA proteins, thus increasing the volume fraction of the
extracellular fluid. As both proteins are
produced by the tip cells, the degradation follows their paths. 

\begin{figure}[ht]
\includegraphics[width=0.32\textwidth]{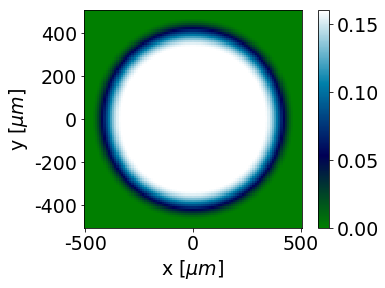} 
\includegraphics[width=0.32\textwidth]{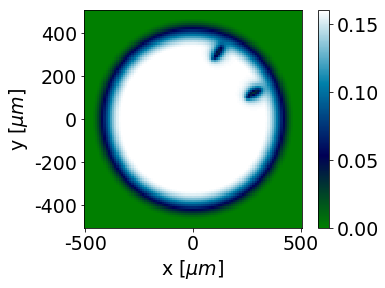}
\includegraphics[width=0.32\textwidth]{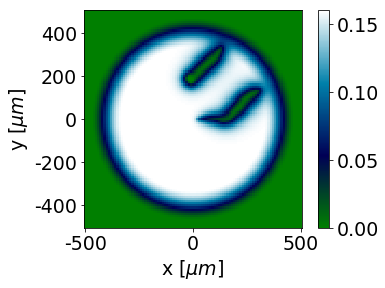}
\includegraphics[width=0.32\textwidth]{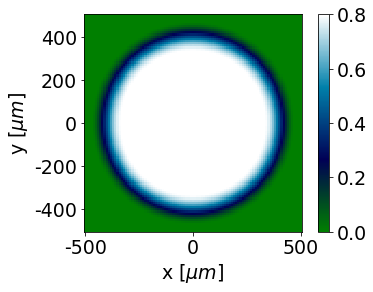} 
\includegraphics[width=0.32\textwidth]{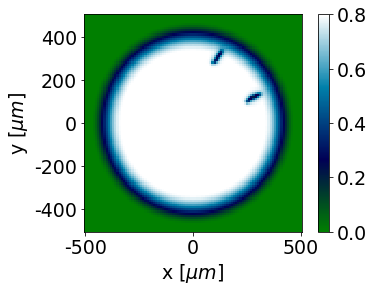}
\includegraphics[width=0.32\textwidth]{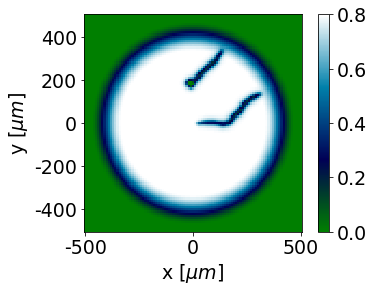}
\includegraphics[width=0.32\textwidth]{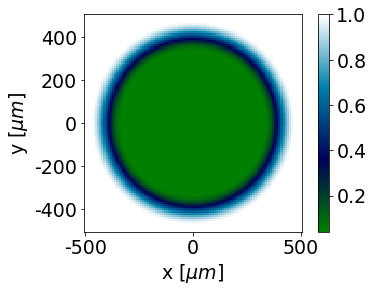} 
\includegraphics[width=0.32\textwidth]{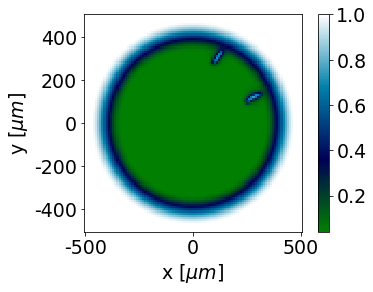}
\includegraphics[width=0.32\textwidth]{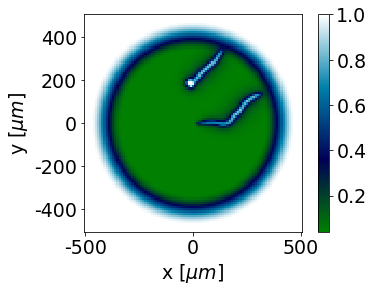}
\caption{Volume fractions of the basement membrane (first row), 
fibrin matrix (second row), and extracellular fluid (last row)
at times $T=0$\,s (left column), $T=400$\,s (middle column), 
and $T=1600$\,s (right column).}
\label{fig.conc}
\end{figure}

Summarizing, we see that the model successfully describes the formation of
premature sprouts. The experiments from \cite{BMGV16} for
dermal endothelial cells show that the in vitro angiogenesis sprouting 
qualitatively well agrees with the numerical tests. Clearly, the proposed
system of equations models only a very small number of biological processes,
chemical reactions, and signal proteins, and more realistic results can be
only expected after taking into account more biological modeling details.
Still, the onset of vessel formation is well illustrated by our simple model.
 

\begin{appendix}

\section{Regularity results for elliptic and parabolic equations}\label{sec.app}

Let $\dom\subset\R^m$ ($m\ge 1$) be a bounded domain.

\begin{theorem}[\cite{Mik76}, Section IV.2, Theorem 4]\label{thm.ellip}
Let $\pa\dom\in C^{k+2}$ and $f\in H^k(\dom)$ for $k\ge 0$. 
Let $u\in H^1(\dom)$ be a weak solution to
$$
  \Delta u = f\quad\mbox{in }\dom, \quad \na u\cdot\nu=0\quad\mbox{on }\pa\dom.
$$
Then $u\in H^{k+2}(\dom)$, and there exists a constant
$C>0$ not depending on $u$ or $f$ such that
$$
  \|u\|_{H^{k+2}(\dom)} \le C\big(\|f\|_{H^k(\dom)} + \|u\|_{L^2(\dom)}\big).
$$
\end{theorem}

The following regularity results hold for the parabolic problem
\begin{equation}\label{a.para}
\begin{aligned}
  & \pa_t u - \diver(a(x,t)\na u) = f\ \mbox{in }\dom,\ t>0, \\
	& a(x,t)\na u\cdot\nu=0\ \mbox{on }\pa\dom, \quad u(0)=u^0\ \mbox{in }\dom.
	\end{aligned}
\end{equation}

\begin{theorem}[\cite{Tem97}, Section II.3, Theorem 3.3]\label{thm.pat}
Let $a\in L^\infty(Q_T)$ be such that $a(x,t)\ge a_0>0$ for all 
$(x,t)\in\overline{\dom}\times[0,T]$, $f\in L^2(Q_T)$, and $u^0\in H^1(\dom)$.
Then there exists a unique weak solution to \eqref{a.para} such that
$u\in C^0([0,T];H^1(\dom))$, $\pa_t u\in L^2(Q_T)$, and there exists a constant
$C>0$, not depending on $a$, $u$, $u^0$, or $f$, such that
$$
  \|u\|_{L^\infty(0,T;H^1(\dom))} + \|\pa_t u\|_{L^2(Q_T)}
	\le C\big(\|f\|_{L^2(Q_T)} + \|u^0\|_{H^1(\dom)}\big).
$$
\end{theorem}

\begin{proof}
The a priori estimate is a consequence of the proof of \cite[Theorem 3.3]{Tem97}.
\end{proof}

Let $\alpha$, $\beta\in(0,1]$.
The space $C^{\alpha,\beta}(\overline{\dom}\times[0,T])$ consists of
all functions $u:\overline{\dom}\times[0,T]\to\R$ such that there exists $C>0$
such that for all $(x,t)$, $y,s)\in\overline{\dom}\times[0,T]$,
$$
  |u(x,t)-u(y,s)|\le C(|x-y|^\alpha + |s-t|^\beta)\quad\mbox{for all }
	(x,t),\,(y,s)\in\overline{\dom}\times[0,T].
$$
The space $C^{k+\beta}(\overline{\dom})$ is the space of all
functions $u\in C^k(\overline\dom)$ such that $\mathrm{D}^ku$ is H\"older continuous 
with index $\beta>0$. 

\begin{theorem}[\cite{Lie87}, Theorem 1.2]\label{thm.hoelder}
Let $\beta\in(0,1)$, $\pa\dom\in C^{1+\beta}$,
$a\in C^{\beta,\beta/2}(\overline{\dom}\times[0,T])$ be such that
$a(x,t)\ge a_0>0$ for all $(x,t)\in\overline{\dom}\times[0,T]$, 
$f\in L^\infty(0,T;L^\infty(\dom))$, and
$u^0\in C^{1+\beta}(\overline\dom)$ be such that $a(x,t)\na u_0\cdot\nu = 0$
on $\pa\dom$. Furthermore, let $u\in C^0([0,T];L^2(\dom))\cap L^2(0,T;H^1(\dom))$
be a weak solution to \eqref{a.para}.
Then there exists a constant $C_\beta>0$, only depending on the data, such that
$$
  \|u\|_{C^{1+\beta,(1+\beta)/2}(\overline{\dom}\times[0,T])}\le C_\beta.
$$
\end{theorem}

\begin{theorem}[\cite{LSU68}, Section IV.9, Theorem 9.1]\label{thm.strong}
Let $\pa\dom\in C^2$, $q>3$, $T>0$, $a\in C^0(\overline{\dom}\times[0,T])$ be such that
$a(x,t)\ge a_0>0$ for all $(x,t)\in\overline{\dom}\times[0,T]$,
$f\in L^q(0,T;L^q(\dom))$, $u^0\in W^{2,q}(\dom)$ be such that
$a(x,t)\na u^0\cdot\nu = 0$ on $\pa\dom$. Then there exists a unique strong solution
$u\in L^q(0,T;W^{2,q}(\dom))$ to \eqref{a.para}
satisfying $\pa_t u\in L^q(0,T;L^q(\dom))$,
and there exists a constant $C>0$, not depending on $u$, $f$, or $u_0$, such that
$$
  \|u\|_{L^q(0,T;W^{2,q}(\dom))} + \|\pa_t u\|_{L^q(0,T;L^q(\dom))}
	\le C\big(\|f\|_{L^q(0,T;L^q(\dom))} + \|u_0\|_{W^{2,q}(\dom)}\big).
$$
\end{theorem}

\begin{theorem}[\cite{LSU68}, Section V.5, Theorem 5.4]\label{thm.class}
Let $\beta\in(0,1)$, $\pa\dom\in C^{2+\beta}$, $T>0$,
$a_{ij}$, $b_i$, $c\in C^{\beta,\beta/2}(\overline{\dom}\times[0,T])$ be such that
$a_{ij}(x,t)\ge a_0>0$ for all $(x,t)\in\overline{\dom}\times[0,T]$ for 
$i,j=1,\ldots,m$, $f\in C^{\beta,\beta/2}(\overline{\dom}\times[0,T])$, and
$u^0\in C^{2+\beta}(\overline\dom)$ be such that $\na u_0\cdot\nu = 0$
on $\pa\dom$. Then there exists a unique classical solution 
$u\in C^{2+\beta,1+\beta}(\overline{\dom}\times[0,T])$ to
\begin{align*}
  & \pa_t u - \sum_{i,j=1}^m a_{ij}(x,t)\frac{\pa^2 u}{\pa x_i\pa x_j}
	+ b(x,t)\cdot\na u + c(x,t)u = f\quad\mbox{in }\dom,\ t>0, \\
  & \na u\cdot\nu = 0\quad\mbox{on }\pa\dom,\ t>0, \quad
	u(0)=u^0\quad\mbox{in }\dom,
\end{align*}
and there exists a constant $C>0$, not depending on $u$, $f$, or $u_0$, such that
$$
  \|u\|_{C^{2+\beta,1+\beta}(\overline{\dom}\times[0,T])}
	\le C\big(\|f\|_{C^{\beta;\beta/2}(\overline{\dom}\times[0,T])}
	+ \|u_0\|_{C^{2+\beta}(\overline\dom)}\big).
$$
\end{theorem}


\section{Model parameters and constants}\label{app.model}

The model parameters and constants are taken from \cite{BMGV16}. For the
convenience of the reader, we collect here the expressions:
\begin{align*}
  \alpha_0 &= \frac{b_{i}R_{c}^3}{F_{i}\mu}, \\
  \gamma(x,t) &= \frac{0.1b_{i}F_{i}(1-f_{E}(x,t))}{\rho_{B}f_{B}(x,t)
  +\rho_{F}f_{F}(x,t)+\rho_{E}f_{E}(x,t)},\\
  \lambda(x,t) &= \frac{4^3 b_{i}F_{i}\widetilde{\lambda}}{30} (1-f_{E}(x,t))\bigg(\frac{1}{2}-f_{E}(x,t)\bigg)f_{E}(x,t), \\
  M_{i}^{k}(x,t) &= \sum_{j=1}^2\sum_{\ell=1}^{N_{j}}
  \frac{F_{i}^2}{20\pi^2 R_{c}^4}(1-f_{E}(x,t))
  \exp\bigg(\frac{-|X_{i}^{k}-X_{j}^{\ell}|}{R_{c}}\bigg) \\
  &\phantom{xx}{}-\frac{2\sqrt{2}}{\pi}\bigg(\frac{\max\lbrace 0,R_{c}-0.5|X_{i}^{k}-X_{j}^{\ell}|\rbrace}{R_{c}}\bigg)^{5/2}, \\
  v_{i}^{k} &= \sum_{j=1}^2\sum_{\ell=1}^{N_{j}}\frac{F_{i}^2}{20\pi^2 R_{c}^4 }(1-f_{E}(x,t))\exp\bigg(\frac{-|X_{i}^{k}-X_{j}^{\ell}|}{R_{c}}\bigg), \\
  z_{i}^{k} & = \frac{v_{i}^{k}}{|v_{i}^{k}|}.
\end{align*}

The parameters are chosen as in the following table; see \cite[Appendix]{BMGV16}.

\medskip
\begin{tabular}{|c|c|c||c|c|c||c|c|c|}
\hline
 & Value & Unit &  & Value & Unit &  & Value & Unit\\
\hline
\hline
$b_i$ & 0.02 & $\textrm{s}^{-1}$ & $D_{V}^{E}$ & 10 & $\mu \textrm{m}^2\textrm{s}^{-1}$ & $r_{D}$ & 10 & $\mu \textrm{m}^3 \textrm{s}^{-1}$\\ 
\hline
$F_{i}$ & 1000 & $\textrm{nN}$ & $D_{D}^{B}$ & 0.51 & $\mu \textrm{m}^2\textrm{s}^{-1}$& $r_{M}$ & 10 & $\mu \textrm{m}^3 \textrm{s}^{-1}$\\
\hline
$\mu$ & 0.2 & --&$D_{D}^{F}$ & 1.02 & $\mu \textrm{m}^2\textrm{s}^{-1}$ & $r_{U}$ & 10& $\mu \textrm{m}^3 \textrm{s}^{-1}$\\
\hline
$\widetilde{\lambda}$ & 15&--& $D_{D}^{E}$ & 0.051 &$\mu \textrm{m}^2\textrm{s}^{-1}$& $s_{V}$ & 0.024& $\mu \textrm{m}^3 \textrm{s}^{-1}$\\
\hline
$R_{c}$ & 11.25 & $\mu \textrm{m} $ & $D_{M}^{B}$ & 1.23 &$\mu \textrm{m}^2\textrm{s}^{-1}$&  $s_{D}$ & 0.024 & $\mu \textrm{m}^3 \textrm{s}^{-1}$\\
\hline
$\rho_{B}$ & $1.06\cdot 10^{-3}$ & $\textrm{ng}\mu \textrm{m}^{-3}$ & $D_{M}^{F}$ & 2.46 &$\mu \textrm{m}^2\textrm{s}^{-1}$ & $s_{M}$ & 0.024 & $\textrm{s}^{-1}$\\
\hline
$\rho_{F}$ & $1.06\cdot 10^{-3}$ & $\textrm{ng}\mu \textrm{m}^{-3}$& $D_{M}^{E}$ & 0.123 &$\mu \textrm{m}^2\textrm{s}^{-1}$& $s_{U}$& 0.024 & $\textrm{s}^{-1}$\\
\hline
$\rho_{E}$ & $0.9933\cdot 10^{-3}$ & $\textrm{ng}\mu \textrm{m}^{-3}$& $D_{U}^{B}$ & 0.53&$\mu \textrm{m}^2\textrm{s}^{-1}$&$s_{B}$ & 1.21 & $\mu \textrm{m}^3 \textrm{ng}^{-1} \textrm{s}^{-1}$\\
\hline
$D_{V}^{B}$ & 100 & $\mu \textrm{m}^2\textrm{s}^{-1}$ & $D_{U}^{F}$ & 1.06 &$\mu \textrm{m}^2\textrm{s}^{-1}$&$s_{F}$ & 1.21 & $\mu \textrm{m}^3 \textrm{ng}^{-1} \textrm{s}^{-1}$\\
\hline
$D_{V}^{F}$ & 200 & $\mu \textrm{m}^2\textrm{s}^{-1}$ &$D_{U}^{E}$ & 0.053&$\mu \textrm{m}^2\textrm{s}^{-1}$&&&\\
\hline
\end{tabular}

\end{appendix}


\end{document}